\documentclass[a4paper,11pt]{amsart}
\usepackage[english]{babel}
\usepackage[utf8]{inputenc}
\usepackage[T1]{fontenc}
\usepackage{csquotes}
\usepackage[style=numeric,
	useprefix,%
	giveninits=true,%
	hyperref,%
	doi=false,%
	url=false,%
	isbn=false,%
	backend=bibtex,%
	maxbibnames=99%
	]{biblatex}
\bibliography{./BIB-FlowsSobolev}

\usepackage{amssymb}
\usepackage{mathrsfs}
\usepackage{thmtools}
\usepackage[usenames,dvipsnames]{xcolor}
\usepackage{enumitem}	
\usepackage{hyperref}
\hypersetup{colorlinks,%
citecolor=Black,%
filecolor=Black,%
linkcolor=Black,%
urlcolor=Black}
\allowdisplaybreaks 
\newcommand{\scr}[1]{\mathscr{#1}}

\newcommand{\bb}[1]{\mathbb{#1}}

\newcommand{\N}{\mathbb{N}}	
\newcommand{\R}{\mathbb{R}}	
\newcommand{\dd}{\,\mathrm{d}}	
\newcommand{\de}{\partial}		
\renewcommand{\div}{{\rm div}}	

\newcommand{\trace}{\mathrm{trace}}
\newcommand{\loc}{\mathrm{loc}}
\newcommand{\spt}{\mathtt{spt}}

\newcommand{\Lbm}{\mathcal{L}} 
\newcommand{\Lbs}{\mathrm{L}} 
\newcommand{\Dom}{\scr D}
\newcommand{\FEL}{\scr E}
\newcommand{\res}{\mathop{\hbox{\vrule height 7pt width .5pt depth 0pt
\vrule height .5pt width 6pt depth 0pt}}\nolimits}

\def\Xint#1{\mathchoice
      {\XXint\displaystyle\textstyle{#1}}%
      {\XXint\textstyle\scriptstyle{#1}}%
      {\XXint\scriptstyle\scriptscriptstyle{#1}}%
      {\XXint\scriptscriptstyle\scriptscriptstyle{#1}}%
      \!\int}
   \def\XXint#1#2#3{{\setbox0=\hbox{$#1{#2#3}{\int}$}
        \vcenter{\hbox{$#2#3$}}\kern-.5\wd0}}
\newcommand{\ave}{\Xint-}

\theoremstyle{plain}
\newtheorem{proposition}{Proposition}[section]
\newtheorem{theorem}[proposition]{Theorem}
\newtheorem{lemma}[proposition]{Lemma}
\newtheorem{corollary}[proposition]{Corollary}
\newtheorem{thm}{Theorem}[section]

\newtheorem*{cor*}{Corollary}

\theoremstyle{definition}
\newtheorem{definition}[proposition]{Definition}
\declaretheorem[name=Remark,sibling=proposition,qed={\raisebox{+0.5ex}{\hbox{\small$\blacklozenge$}}}]{remark}

\title[Flows of Sobolev vector fields]{Classical flows of vector fields with exponential or sub-exponential summability}

\date{\today}

\author[Ambrosio]{Luigi Ambrosio}
\address[Ambrosio]{Scuola Normale Superiore, Pisa, Italy.}
\email{luigi.ambrosio@sns.it}

\author[Nicolussi Golo]{Sebastiano Nicolussi Golo}
\address[Nicolussi Golo]{Department of Mathematics and Statistics, University of Jyväskylä, Finland}	
\email{sebastiano.s.nicolussi-golo@jyu.fi}

\author[Serra Cassano]{Francesco Serra Cassano}
\address[Serra Cassano]{Dipartimento di Matematica, Università di Trento, Italy}
\email{francesco.serracassano@unitn.it}

\thanks{Corresponding author: S.~Nicolussi Golo.}

\keywords{Vector Fields, Flow, Sobolev--Orlicz Spaces, Transport Equation, Continuity Equation}

\subjclass[2010]{%
35F10  
35A01  
35A02  
}

\begin{document}
\begin{abstract}
We show that vector fields $b$ whose spatial derivative $D_xb$ satisfies a Orlicz summability condition 
 have a spatially continuous representative and are well-posed.  For the  case of sub-exponential summability, their flows satisfy a Lusin (N) condition in a quantitative form, too.
Furthermore, we prove that if 
$D_xb$ satisfies a suitable exponential summability condition 
then the flow associated to $b$ has Sobolev regularity, without assuming boundedness of $\div_xb$.
We then apply these results to the representation and Sobolev regularity of 
weak solutions of the Cauchy problem for the transport and continuity equations.
\end{abstract}
\maketitle

\setcounter{tocdepth}{3}
\phantomsection
\addcontentsline{toc}{section}{Contents}
\tableofcontents


\section{Introduction}
In this paper we are concerned with the study of the existence and uniqueness of classical solutions of the Cauchy problem for
the ODE system
\begin{equation}\label{eq03161620}  
\left\{
\begin{array}{r@{\,=\,}l}
\dot{\gamma}(t)& b(t,\gamma(t))\\
\gamma(s)& x\,,
\end{array}
\right.
\end{equation} 
with $x\in\Omega$, an open domain in $\R^n$, $s\in I$, an open interval in $\R$, and $b:\,I\times\Omega\to\R^n$  a  
 continuous,  possibly non-autonomous vector field 
Even though we will mostly deal with the case when $b$ is continuous,
we will point out which proofs can easily be adapted to the case when $b$ is only measurable with respect to $t$.\
If solutions to~\eqref{eq03161620} exist and are unique for every $s$ and $x$, we say that the vector field $b$ is \emph{well-posed} in $I\times\Omega$ (or in $\overline I\times\Omega$, see Definition~\ref{WP} 
for a more precise statement). For every well posed vector field $b:\,I\times\Omega\to\R^n$ we have a \emph{flow}, that is, a map
$X:I\times I\times\Omega\to\Omega$, defined as  $X(t,s,x):=\gamma(t)$ where $\gamma$ is the unique absolutely continuous solution of~\eqref{eq03161620}.
More precisely, for each $t,\,s\in I$ we denote by $\Omega_{(t,s)}\subset\Omega$ the open set of all $x\in\Omega$ such that the path starting at $x$ at time $s$ can be extended until time $t$ (see  Section \ref{sec5ed296d6} and Remark~\ref{rem62a45ceb}).
Then $X(t,s,\cdot)$ is a well defined homeomorphism $\Omega_{(t,s)}\to\Omega_{(s,t)}$ (see Remark~\ref{rem62a45ceb}).

Di~Perna--Lions~\cite{MR1022305}  carried out a pioneering and far-reaching theory by introducing a generalized notion of flow for vector fields $b\in L^1_\loc((0,T) ;W^{1,1}_\loc(\R^n ,\R^n))$ with important applications to the existence and uniqueness of weak solutions for the Cauchy problem of  the \emph{transport  equation} associated to a weakly differentiable vector field $b$, that is, 
\begin{equation}\label{CPTE}
\begin{cases}
\partial_t u+b\cdot D_x u= 0 &\text{ in $(0,T)\times\R^n$ }\\
u(0,\cdot)=\bar u\,.
\end{cases}
\end{equation}
The theory was later remarkably extended by the first author~\cite{MR2096794} to vector fields $b(t,\cdot)$ with $BV$ regularity. 
In these works, the regularity of $b$ is paired with the boundedness of its spatial divergence, that is
\begin{equation}\label{divbbd}
\div_x b \in L^1((0,T) ; L^\infty(\R^n )) ,
\end{equation}
which ensures the existence and uniqueness of the generalized flow of $b$.
If~\eqref{divbbd} does not hold, then uniqueness of the flow may fail, 
as it was already shown in \cite[Section IV.1]{MR1022305}. 
The existence and uniqueness of a  generalized flow associated to a weakly regular vector field $b$ has been the object of an intensive study with applications  to the Cauchy problem for the transport equation as well as for the \emph{continuity equation} associated to $b$, that is,
\begin{equation}\label{CPCE}
\begin{cases}
\partial_t\rho+\div_x(b\rho) = 0 & \text{ in $(0,T)\times\R^n$ }\\
\rho(0,\cdot) = \bar\rho &\text{ in $\R^n$\ . }
\end{cases}
\end{equation}
Existence, uniqueness and regularity of solutions of these three problems~\eqref{eq03161620},~\eqref{CPTE} and~\eqref{CPCE} are connected with each other.
In particular, the existence of a unique flow $X$ with enough regularity implies existence and uniqueness of solutions to both the transport equation and the continuity equation. A fairly complete account of the development  in this topic can be found in \cite{MR3283066} and references therein.
A sample of the literature on this subject is \cite{MR3375545,MR4263701,MR2181034,MR3778580,MR3912727,MR3494396,MR3458193,MR3906270,MR3479536,MR1421208,MR4061966,MR2679007,MR2044334}.

\bigskip

Our contribution focuses on two problems.
First, 
we want to weaken the boundedness assumption
on the divergence~\eqref{divbbd}.
We will show in Theorem~\ref{thm5e999915} that sub-exponential summability of $\|D_xb\|$ guarantees the existence of a unique \emph{classical} flow (in the Di Perna--Lions--Ambrosio theory, flows have a weaker definition).

Second, we want to find conditions on $b$ for the flow to have Sobolev regularity, 
instead of just $L^p$ integrability.
It is well-known that  high $L^p$ integrability of matrix Jacobian $D_xb$, even coupled with~\eqref{divbbd}, is not enough in order to provide Sobolev regularity of the flow $X$ (see, for instance, \cite{MR3437603}).
 A strategy used in the recent papers \cite{MR3912727,MR4263701} was to strengthen the hypotheses by requiring  exponential summability of $\|D_xb\|$.
We refer in particular to the recent paper \cite{MR4263701}, where it has been shown that that $b$ has a unique  flow with Sobolev regularity under the condition
\begin{equation}\label{eq5edfe619}
\sup_{t\in\R} \int_{\bb T^n} \exp(\beta \|D_xb(t,x)\|) \dd x < \infty
\quad\text{and}\quad
\div_xb\in\Lbs^\infty_\loc(\R\times\bb T^n) ,
\end{equation}
for some $\beta>0$, where $\bb T^n$ is the $n$-dimensional torus. 
We prove analogous results without conditions on the divergence of $b$ in Theorem~\ref{thm5e9a10ec}, see also Remark~\ref{rem:torus}.

\bigskip

Our results are of three types.
We first provide integral conditions of sub-exponential type on $Db$ that ensure well-posedness.
Then, we study the Sobolev regularity of the homeomorphisms $X(t,s,\cdot)$.
Finally, we apply these results to both the transport equation and the continuity equation.

\subsection{Well-posedness}
Let us focus, first, on the well-posedness. If $\left\|D_xb\right\|$ satisfies an \emph{exponential summability},
that is, conditions of the form
\begin{equation}\label{eq5eb965b3}
\int_I\int_\Omega\exp\left(\beta \|D_xb(t,x)\| \right)\dd x \dd t<+\infty
\end{equation} 
for some $\beta>\,0$, then it is well-known that $b$ is well-posed.
Indeed, in this case, the vector field $b(t,\cdot)$ satisfies a so-called Log-Lipschitz condition;
see, for instance, \cite{MR4263701,MR1920427}.
However, reformulating the condition of exponential summability 
in a Orlicz-like form, we extend the result to some sub-exponential cases.

\begin{thm}\label{thm5e999915}
Let $\Theta:[0,+\infty)\to(0,+\infty)$ be a non decreasing locally Lipschitz function.
Assume that
\begin{enumerate}[label=(\ref{thm5e999915}.\Roman*)]
\item\label{convTheta}  
if $n>\,1$, there exists $\alpha\in(1,\frac{n}{n-1})$ such that $\Theta^{\frac{\alpha-1}{\alpha}}$ is convex, 
while, if $n=1$, $\Theta$ is convex;
\item\label{item04021111}
there exists $C_\Theta\ge1$ such that $\Theta:\,[C_\Theta,+\infty)\to [\Theta(C_\theta),+\infty)$ is bijective and
\begin{equation}\label{condBI}
\Theta(s_1)\Theta(s_2) \le \Theta(C_\Theta\, s_1\,s_2)\text{ for all }s_1,s_2\ge C_\Theta\,;
\end{equation}
\item\label{item04021113}
$\displaystyle{\int_1^\infty\frac{\Theta'(s)}{s\Theta(s)} \dd s = +\infty}$.
\end{enumerate}
Let $b\in \Lbs^1_\loc(I;W^{1,1}_\loc(\Omega;\R^n))$ and assume that
for every $o\in\Omega$ there exist $c>0$, $R>0$  such that $B(o,2R)\subset\Omega$ and
the function
\begin{equation}\label{eq5ea991ab}
t\mapsto \psi(t):=\,\int_{B(o,2R)} \Theta(c\|D_x b(t,z)\|) \dd z
\end{equation}
belongs to $\Lbs^1_\loc(I)$. 
Then $b(t,\cdot):\Omega\to\R^n$ has a continuous representative
$\tilde{b}(t,\cdot)$ for a.e.~$t\in I$ and $\tilde{b}$ is well-posed in $I\times\Omega$.
Moreover, if  there exists $m\in L^1(I)$ such that
\begin{equation}\label{estimbuniftA}
|b(t,x)|\le\,m(t)\quad\text{ for a.e.~}t\in I,  \text{ for a.e.~}x\in B(o,R)\,,
\end{equation}
then $\tilde{b}$ is also well-posed in $\overline{I}\times\Omega$.
\end{thm}

Notice that a byproduct of the proof of Theorem~\ref{thm5e999915}
is that the Sobolev-Orlicz space $W^1L_\Theta(\Omega)$ embeds in $C^0(\Omega)$, with modulus of continuity that depends only on $\Theta$. 
See \cite[Section 2.6]{ANGSC} for the definition of $W^1L_\Theta(\Omega)$.

The proof is inspired from \cite{MR1920427}.
In fact, we shall prove that a continuous representative $\tilde b(t,\cdot)$ of  $b(t,\cdot)$ satisfies the Osgood's criterion (see~\cite[Chap. III, Corollary 6.2]{MR1929104} or Proposition~\ref{prop628f93c2} below). The proof of this result is given in Section~\ref{sec5edff2b1}.

We note in Proposition~\ref{prop627d11fd} that if an increasing function $\Theta$ 
satisfies condition~\ref{item04021113} of Theorem~\ref{thm5e999915}, 
then it does not have polynomial growth.
Examples of functions $\Theta$ satisfying the properties~\ref{convTheta}-\ref{item04021113}
are of the form
\begin{equation}\label{sub2expofunct}
\FEL_{k,\beta}(s) = \exp\left(\frac{s}{\log(s)\log\log(s)\dots(\underbrace{\log\dots\log }_{\text{ $k$-times}}s)^\beta} \right)
\quad\text{ for }s \ge\bar s ,
\end{equation}
with $\FEL_{k,\beta}(s)=\FEL_{k,\beta}(\bar s)$ for $s<\bar s$, where $\bar s$ is  large enough, $k\ge\,1$ is an integer and $0\le \beta\le1$, 
see Proposition~\ref{propFEL}. Notice that the asymptotic behaviour of $\FEL_{k,\beta}$ as $s\to\infty$ is almost sharp in order that assumption~\ref{item04021113} holds, see Remark~\ref{sharpbehavsubexpo}.
If $\FEL_{k,1}(c\left\|D_xb\right\|)\in \Lbs_\loc^1(\Omega)$
for some $c>0$,  we say that $\left\|D_xb\right\|$ satisfies a \emph{subexponential summability} of order $k$.
Therefore, Theorem~\ref{thm5e999915} shows that, if $\left\|D_xb\right\|$ has \emph{subexponential summability}, then $b$ has a \emph{classical} unique flow. 
However, we stress that, under the hypothesis of Theorem~\ref{thm5e999915}, $\left\|D_xb\right\|$ does not need not be in $L^p_{\rm loc}(I\times\Omega)$ for each $p>1$ (see Remark~\ref{rem628c97ad}).

\subsection{Regularity}
Moving on to the regularity of the flow,  we can prove that, 
if $D_xb$ satisfies a subexponential summability of order $1$, the associated flow $X(t,s,\cdot)$ 
satisfies a weak regularity property, namely it maps the Lebesgue measure into absolutely continuous measures. 
Notice that, in this case, $\left\|D_xb(t,\cdot)\right\|$ does belong to $L^p_{\rm loc}(\Omega)$ for every $p>1$, for almost every $t\in I$.
A quantitative version,
that we obtain adapting \cite{MR3494396}, is the following.

\begin{thm}\label{thm5eb900db}
Let $b\in \Lbs^1_\loc(I;W^{1,1}_\loc(\R^n;\R^n))$.
Suppose that
\begin{equation}\label{globexsistb}
\frac{|b(t,x)|}{1+|x|\,\log^+|x|}\in L^1(I;L^\infty(\R^n))\,,
\end{equation}
and
\begin{equation}\label{eq5eb95263}
\int_I\int_{\R^n}\exp\left( \frac{ \|D_xb\| }{ \log^+\|D_xb\| } \right)(t,x)\,d\gamma_n(x)dt <\,+\infty,
\end{equation}
where  $\gamma_n$ is the Gaussian measure on $\R^n$,
namely
\[
\gamma_n:=\,\frac{1}{(2\pi)^{n/2}}\exp\left(-\frac{|x|^2}{2}\right) \Lbm^n\, 
\]
with $\Lbm^n$ the Lebesgue measure in $\R^n$.

Then the space continuous representative $\tilde{b}$, granted by the Sobolev embedding, is well-posed in $\overline{I}\times\Omega$ and the associated flow $X$ of $\tilde{b}$ is globally defined, that is,  
$X:\overline I\times \overline I\times\R^n\to\R^n$.
Moreover, for every $t,\,s\in \overline I$, the image measure $X(t,s,\cdot)_\# \Lbm^n$ is absolutely continuous with respect to $\Lbm^n$ 
and there exists a positive constant $\alpha_0(t,s)>0$ such that
\begin{equation*}\label{eq5eb91e42}
\frac{\dd}{\dd \Lbm^n}(X(t,s,\cdot)_\# \Lbm^n)\in \Lbs^{\Phi_\alpha}_\loc(\R^n)
\end{equation*}
for each $0<\alpha<\alpha_0(s,t)$, where $\Lbs^{\Phi_\alpha}_\loc(\R^n)$ is the Orlicz space with
\begin{equation*}\label{Phialpha}
\Phi_\alpha:[0,+\infty)\to[0,+\infty),
\qquad
\Phi_\alpha(w):= w \exp((\log^+ w)^\alpha) ,
\end{equation*}
and $\alpha_0:I\times I\to\R$ is continuous with $\alpha_0(t,t)=1$ for every $t\in I$.
\end{thm}

See Section~\ref{sec5edfede8} for the proof. 
By trivial considerations, in the one-dimensional case, we can improve Theorem~\ref{thm5eb900db} to absolutely continuity of the flow.

\begin{thm}\label{thm5ea9862b}
If $n=1$ and $b$ satisfies the conditions of Theorem~\ref{thm5eb900db}, then, for every $t,\,s\in I$,
the map $X(t,s,\cdot)$ is an absolutely continuous homeomorphism between intervals of~$\R$.
\end{thm}

Sobolev regularity stated in Theorem~\ref{thm5ea9862b} is sharp, as we show in Example~\ref{sec5eb962e5}.
We don't know whether Theorem~\ref{thm5ea9862b} can be extended to
the case of sub-exponential summability of order $k>1$. 

In higher dimensions, for subexponential summability,
we have a partial negative result with an example in Section~\ref{sec5eb962e5},
or the example constructed in~\cite{MR4263701}, see also Remark~\ref{rem5edfebb7}:
{\it there are vector fields satisfying a subexponential summability of order $1$ whose flow is not in $W^{1,p}_\loc$ for 
 any $p>n$}. 
However, it remains open whether vector fields satisfying a subexponential summability of order $1$ can fail to have the flow in $W^{1,1}_\loc$.

On the other hand, in higher dimensions, for exponential summability we have a positive result:

\begin{thm}\label{thm5e9a10ec}
Let $I\subset\R$ and $\Omega\subset\R^n$ be bounded open sets and let $b\in \Lbs^1_\loc(I;W^{1,1}_\loc(\Omega;\R^n))$ be bounded.
Assume that for some $p>2n$ the vector field $b$ satisfies the global geometric condition
\begin{equation}\label{eq:geometric}
\Lambda_p:=\int_\Omega\int_I \max\left\{\ell^{\frac{n}{n-p}},\frac{({\rm dist}(x,\partial\Omega))^{\frac{n}{n-p}}}{(\sup |b|)^{\frac{n}{n-p}}}\right\}
	\exp\left(\frac{\ell p^2}{p-n} \|D_xb(s,x)\| \right)\dd s\dd x<+\infty,
\end{equation}
with $\ell$ equal to the length of $I$. 
 Then the space-continuous representative $\tilde{b}$ is well-posed in $\overline{I}\times\Omega$
thanks to Theorem~\ref{thm5e999915} and, 
in addition, for every $t\in {I}$ and for almost every $s\in I$, one has
\begin{equation}\label{eq5ff85eab}
\text{for a.e.~}s\in I,\quad
X(t,s,\cdot)\in W^{1,p}(\Omega_{(t,s)};\R^n)
\quad\text{and}\quad
X(s,t,\cdot)\in W^{1,p}(\Omega_{(s,t)};\R^n) ,
\end{equation}
with
\begin{equation}\label{eq63eb5e33}
	\int_I\int_{\Omega_{(t,s)}} \|D_xX(t,s,x)\|^p \dd x\dd s\leq \ell^{\frac{n}{p-n}}\Lambda_p .
\end{equation}
\end{thm}

With regards to~\eqref{eq5ff85eab}, 
a lower Sobolev regularity for the flow can be proved for all pairs of times $s,t\in I$, see Corollary~\ref{cor63edf816}.
For $b\in C^0(I;C^1(\Omega;\R^n))$, we have a more detailed statement in Theorem~\ref{thm5e9a10ec_bis}.
We may also consider the case when $b:I\times\R^n\to\R^n$ and the support of $b(t,\cdot)$ is contained in a 
compact set  independent of $t$. See also Remarks~\ref{rem:torus} and~\ref{rem63eb57a1}.

\begin{thm}\label{thmbcpcsupp}
	Let $I\subset\R$ be a bounded open interval and let $b\in \Lbs^1_\loc(I;W^{1,1}_\loc(\R^n;\R^n))$
	be a bounded vector field. Assume that there exist a bounded open set $\Omega\subset\R^n$ and $p>n$  such that
	\begin{equation}\label{eq63eb50aa}
	\spt(b(t,\cdot))\subset \Omega\quad\text{for each $t\in I$}
	\end{equation}
	and \eqref{eq5eb965b3} holds with $\beta=\,{\ell p^2}/{(p-n)}$ and $\ell$ equal to the length of $I$. 	
	
	Then the space-continuous representative $\tilde{b}$, understood as vector field in $I\times\Omega$, is well-posed in $\overline{I}\times\Omega$
	thanks to Theorem~\ref{thm5e999915}, $\Omega_{(t,s)}=\,\Omega$ and  for every $t,\,s\in \overline{I}$ one has 
	\begin{equation}\label{eq5e9998eb}
	\begin{split}
	&X(t,s,\cdot)\in W^{1,p}(\Omega;\R^n)\quad\text{with}\\
	&
	\int_{\Omega} \|D_xX(t,s,x)\|^p \dd x
	\leq
	\frac1{\ell} 
	\int_{I}\int_\Omega 
		\exp\left( \frac{\ell p^2}{p-n}\|D_xb(v,y)\| \right) 
	\dd y\dd v .	
	\end{split}
	\end{equation}
\end{thm}

Notice also that, in the previous theorem, we obviously have that the flow map $X:\,\overline{I}\times\overline{I}\times\R^n\to\R^n$ of $b$ is 
identically equal to the identity on $\overline{I}\times\overline{I}\times (\R^n\setminus\overline{\Omega})$.

\begin{remark} \label{rem:extensions} 
(1) Since ${\rm dist}(x,\partial\Omega)$ is bounded, 
\eqref{eq:geometric} can be stated in the equivalent form
$$
\int_\Omega\int_I ({\rm dist}(x,\partial\Omega))^{n/(n-p)}\exp\left(\frac{\ell p^2}{p-n} \|D_xb(s,x)\| \right)\dd s\dd x<+\infty.
$$
We used that specific form for the purpose of the estimate \eqref{eq5ff85eab}.

(2) If $\alpha=n/(p-n)$ is smaller than 1 (i.e., $p>2n$) and $\partial\Omega$ is regular, then the geometric condition~\eqref{eq:geometric}, thanks to the
H\"older inequality, is implied by the simpler condition
\begin{equation}\label{eq:flow63}
\int_\Omega\int_I \exp\left(c\|D_xb(s,x)\| \right)\dd s\dd x<+\infty
\end{equation} 
provided $c>(1/\alpha)'p^2/(p-n)$, where $(1/\alpha)'=(1-\alpha)^{-1}$ is the dual exponent.
\end{remark}

We also point out that the Sobolev exponent $p$, related to the constant in the exponential
integrability condition in \eqref{eq:geometric}, may not be sharp. 
In other words, we can have $\Lambda_p=\infty$ but $X(t,s,\cdot)\in W^{1,p}(\Omega_{(t,s)};\R^n)$, see Example~\ref{sec5ec2a1ee}.

Several regularity properties of $X$ follow from Theorem~\ref{thm5e9a10ec}, see Corollaries~\ref{cor5ec7b067},~\ref{cor5edfe500} and~\ref{cor5ec2a4b4}. In particular we show that, as in Theorem~\ref{thm5eb900db},  $X(t,s,\cdot)_\#\Lbm^n\ll\Lbm^n$, for each 
$t,\,s\in\overline{I}$.

\begin{remark}
	Suppose that $b\in \Lbs^1_\loc(I;W^{1,1}_\loc(\R^n;\R^n))$
	has compact support and~\eqref{eq:flow63} holds for some $c>0$.
	For every $p>2n$, we get~\eqref{eq5e9998eb} on subintervals of $I$ of length $\ell$ such that
	$\frac{\ell p^2}{p-n} = c$.
	Since $\frac{p^2}{p-n}>4n$ for $p>2n$, then $\ell<\frac{c}{4n}$.
\end{remark}

\subsection{Transport and continuity equations}
Finally, we apply Theorem~\ref{thm5eb900db} and Theorem~\ref{thm5e9a10ec} to the transport equation~\eqref{CPTE} and the continuity equation~\eqref{CPCE}, respectively. By standard methods, we provide existence, uniqueness, representation and regularity of solutions.
 
In \cite[Theorem 1]{MR3458193},  uniqueness of weak solutions $u\in L^\infty((0,T);L^\infty(\R^n))$
has been proved, provided that the spatial derivatives of $b$ satisfies a   sub-exponential summability and $\bar u\in L^\infty(\R^n)$.
Our next result provides even in this context the classical representation of this unique solution in terms of the flow map,
granted by Theorem~\ref{thm5eb900db}.

\begin{thm}[Representation of the solutions for the transport equation]\label{solTE}
Let $b$ be a vector field as in Theorem~\ref{thm5eb900db}, and let $X$ be the flow associated to the continuous
representative $b$.
Then, for each $\bar u\in L^\infty(\R^n)$, for each $t\in [0,T]$, the function
\begin{equation}\label{represcharTE}
v(t,x):=\,\bar u\left(X(t,0,\cdot)^{-1}(x)\right)\text{ for a.e.~}x\in\R^n\,,
\end{equation}
is the unique weak solution in $L^\infty((0,T);L^\infty(\R^n))$ of the Cauchy problem for the transport equation~\eqref{CPTE},
understood in the sense of distributions.
\end{thm}

By means of the theory of maps with finite distortion, see for instance~\cite{MR3184742},
we can also show Sobolev regularity of the solution $v$ to the transport equation,
see Corollary~\ref{cor5ec7d875}. 

By applying  the well-posedness and representation results proved in \cite{MR2439520,MR3906270} (see Theorem \ref{wpCE} and Remark \ref{rmkclop}),  we make more explicit the representation of the weak solutions of \eqref{CPCE}, under the assumptions of 
Theorem~\ref{thm5eb900db} and Theorem~\ref{thmbcpcsupp}. It is useful to deal with the case  where $\rho(t,\cdot)$ belongs to the space of signed Borel measures on $\R^n$ with finite total variation, which we will denote by $\mathcal M(\R^n)$.

\begin{thm}[Representation of the solutions for the continuity equation]\label{solCE} 
Let $I=(0,T)$, let $b:I\times\R^n\to\R^n$ be a vector field with $b(t,\cdot)$ continuous for a.e.~$t\in I$.
\begin{itemize}
\item [(i)] Suppose that $b$ satisfies the assumptions of Theorem~\ref{thm5eb900db}  and let $X$ be the flow associated to $b$. Then, for each signed measure $\bar\rho\in \mathcal M(\R^n)$, 
\begin{equation}\label{solCPCEmeas}
\rho_t=\,\rho(t,\cdot)=\,X(t,0,\cdot)_\#\bar\rho \quad t\in [0,T]
\end{equation}
is the unique weak solution for the Cauchy problem \eqref{CPCE} in $L^\infty((0,T);\mathcal M(\R^n))$.
		
Moreover, if $\bar\rho\in L^1(\R^n)$, $\rho_t$ can be represented as 
\begin{equation}\label{eq6033caa8}
\rho_t:= (\bar\rho J_{X,t}) \circ X(0,t,\cdot)\, ,
\end{equation}
where $J_{X,t}(y) = \frac{\dd X(t,0,\cdot)_\#\mathcal L^n }{\dd \mathcal L^n}(X(t,0,y))$.
In particular we have $\rho\in L^\infty((0,T);L^1(\R^n))$.
\item[(ii)] If $b$ satisfies the stronger hypotheses of Theorem~\ref{thmbcpcsupp} and $\bar\rho\in L^1(\R^n)$, then  
 the unique  weak solution for  the Cauchy problem ~\eqref{CPCE} can also be represented as
\begin{equation}\label{reprsolCE}
\rho_t=\,\frac{\bar\rho}{J_{X(t,0,\cdot)}}\circ X(0,t,\cdot)  \,,
\end{equation}
\end{itemize}
where $J_{X(t,0,\cdot)}$  denotes the Jacobian of homeomorphism $X(t,0,\cdot):\,\R^n\to\R^n$ (see \eqref{jacob}).
\end{thm}

Notice that existence and uniqueness of distributional solutions of continuity equation \eqref{CPCE} has been obtained in \cite{MR3479536} by weakening assumption \eqref{divbbd}.
More precisely, the authors showed that if $\bar\rho\in L^\infty(\R)$ and
\begin{equation}\label{eq5fc0c25e}
\div_x b \in (BMO\cap L^1)(\R^n),
\end{equation}
then there exists a unique distributional solution $\rho\in L^\infty((0,T)\times\R^n)$ of the Cauchy problem \eqref{CPCE}
However, we point out that condition \eqref{eq5eb965b3} does not imply \eqref{eq5fc0c25e}, see Example~\ref{sec5ec2a1ee}.

\subsection*{Structure of the paper}
We start in Sections~\ref{sec5edff236} and~\ref{sec5edff28b} by listing known results about homeomorphisms of finite distortion
and classical flows of vector fields.
In Section~\ref{sec5edff2b1} we prove Theorem~\ref{thm5e999915}. Theorems~\ref{thm5eb900db} and~\ref{thm5ea9862b} are proven in Section~\ref{sec5edfede8}. 
The proof of Theorem~\ref{thm5e9a10ec} is given in Section~\ref{sec5edff316}.
Theorems~\ref{solTE} and~\ref{solCE} are shown in Section~\ref{sec5edff3f5}.
Finally,  Section~\ref{sec5edff41e} is devoted to the illustration of a few examples.

\subsection*{Acknowledgements} 
We thank A.~Clop for useful suggestions and discussions on this topic. We are also grateful to R.~Serapioni for several discussions about Section~\ref{sec5edff2b1}. 
We want to thank also the anonymous referee for their useful comments and corrections.

L.A.~and F.S.C.~have been supported by the PRIN 2017 project ``Gradient flows, Optimal Transport and Metric Measure Structures.''
S.N.G.~has been supported
	by University of Padova STARS Project ``Sub-Riemannian Geometry and Geometric Measure Theory Issues: Old and New''
	and by the Academy of Finland (%
grant 328846 ``Singular integrals, harmonic functions, and boundary regularity in Heisenberg groups'',
grant 322898 ``Sub-Riemannian Geometry via  Metric-geometry and Lie-group Theory'',
grant 314172 ``Quantitative rectifiability in Euclidean and non-Euclidean spaces'').
	S.N.G.~and F.S.C.~have been supported 
	by the INdAM – GNAMPA Project 2019 ``Rectifiability in Carnot groups''.

\section{Preliminaries on homeomorphisms}\label{sec5edff236}

In this section we present some results about homeomorphisms that we will need later.
Most of the material comes from the book~\cite{MR3184742}.

\subsection{Weak derivatives of homeomorphisms}

If $\Omega\subset\R^n$ is open and $\Psi:\Omega\to\R^n$, we denote by $D\Psi$ the weak differential of $\Psi$.
When a point-wise analysis is needed, we assume that $D\Psi(x)$ is equal to the classical differential of $\Psi$ at every $x\in\Omega$ where $\Psi$ is differentiable.\footnote{Notice that 
the distributional derivative of a Sobolev function $\Psi$ agrees with the classical derivative at a.e.~$x$ where $\Psi$ is differentiable.
If the reader needs an argument, let us point out that the approximate differential of $\Psi$ agrees with the distributional derivative almost everywhere by~\cite[Theorem~6.1.4, page~233]{MR3409135}, and that the classical derivative is equal to the approximate differential wherever it exists.}
The \emph{Jacobian} of $\Psi$ is  
\begin{equation}\label{jacob}
J_\Psi(x) = \det(D\Psi(x))\,.
\end{equation}
A map $\Psi:\Omega\to\R^n$ satisfies the \emph{Lusin (N) condition} if for all $E\subset\Omega$ the vanishing of $\Lbm^n(E)$ implies $\Lbm^n(\Psi(E))=0$.
An \emph{embedding} is a homeomorphism onto its image.

\begin{lemma}[{\cite[Lemma A.29]{MR3184742}}]\label{lem5e9a12a0}
	Let $\Omega\subset\R^n$ be open  and let $\Psi:\Omega\to\R^n$ be an embedding that is differentiable at $x$.
	The map $\Psi^{-1}$ is differentiable at $\Psi(x)$ if and only if $\det(D\Psi(x))\neq0$.
	In this case, we have
	\[
	D(\Psi^{-1})(\Psi(x)) = (D\Psi(x))^{-1} .
	\]
\end{lemma}

\begin{lemma}[{\cite[Theorem A.35]{MR3184742}}: Area Formula]\label{lem04031003}
	If $\Omega\subset\R^n$ is open, $\Psi\in W^{1,1}_\loc(\Omega;\R^n)$ and $\eta:\R^n\to[0,+\infty]$ is a Borel function, then 
	\begin{equation}\label{eq03161203}
	\int_\Omega \eta(\Psi(x)) |J_\Psi(x)| \dd x \le \int_{\R^n} \eta(y) N(\Psi,\Omega,y) \dd y ,
	\end{equation}
	where $ N(\Psi,\Omega,y) = \# (\Psi^{-1}(y)\cap\Omega)$ is the cardinality of the set $\Psi^{-1}(y)\cap\Omega$.
	If $\Psi$ satisfies also the Lusin (N) condition then equality holds in~\eqref{eq03161203}.
\end{lemma}

\begin{remark}  Observe that  the Lusin (N) condition for a homeomorphism $\Psi:\Omega\to\Omega'$  is equivalent to assume that the push-forward measure $(\Psi^{-1})_\#(\mathcal L^n\res\Omega')$ is absolutely continuous with respect to  $\mathcal L^n$.  
An abstract area-type formula with respect to a general measure $\mu$ is also introduced in Lemma \ref{represpfmeas}.
\end{remark}

\begin{lemma}[{\cite[Theorem 4.2 and Theorem 6.2.1]{MR3184742}}]\label{lem5e9a1230}
	If $\Omega\subset\R^n$ is open, $\Psi\in W^{1,p}_\loc(\Omega;\R^n)$ and $p>n$, then  the continuous
	representative $\tilde{\Psi}$ is differentiable at 
	a.e.~$x\in\Omega$ and $\tilde\Psi$ satisfies the Lusin (N) condition.
\end{lemma}

\subsection{Mappings of finite distortion}

We denote by $\|M\|$ the operator norm of a linear map  $\R^n\to\R^n$ or $m\times n$-matrix, that is, 
\[
\|M\|:= \sup \left\{|Mv|: v\in\R^n,\ |v|=1 \right\} ,
\]
where $|\cdot|$ is the Euclidean norm. With this choice of the norm,
one can easily show the following inequalities (the first one is called \emph{Hadamard's inequality}):
\begin{equation}\label{eq03161234}
|\det M| \le \Pi_{i=1}^n |Me_i| \le \|M\|^n  .
\end{equation}

\begin{definition}
	A \emph{map of finite distortion} on an open set $\Omega\subset\R^n$ is a function $\Psi\in W^{1,1}_\loc(\Omega;\R^n)$ such that there exists $K:\Omega\to[1,+\infty)$ with
	\begin{equation}\label{eq03161220}
	\|D\Psi(x)\|^n \le K(x) J_\Psi(x)
	\qquad\text{for a.e.~}x\in\Omega.
	\end{equation}
	For $q\ge1$, the \emph{$q$-distortion function} of a map of finite distortion $\Psi$ is
	\[
	K_q^\Psi(x) := 
	\begin{cases}
	 	\displaystyle{\frac{\|D\Psi(x)\|^q}{J_\Psi(x)}} & 	\text{ if }J_\Psi(x) \neq 0 , \\
		1&\text{ otherwise }.
	\end{cases}
	\]
\end{definition}
Our main reference on maps of finite distortion is~\cite{MR3184742}.
Notice that $K_n^\Psi$ is the optimal distortion function for the inequality~\eqref{eq03161220} to hold. 

\begin{lemma}\label{lem5ec7e29c}
	If $\Omega\subset\R^n$ is open and $\Psi\in W^{1,1}_\loc(\Omega;\R^n)$ is an embedding with $J_\Psi > 0$ almost everywhere in $\Omega$, 
	then $\Psi$ has finite distortion.
\end{lemma}
\begin{proof}
	The function $K^\Psi_n$ is finite almost everywhere and~\eqref{eq03161220} holds whenever $J_\Psi(x)>0$.
	Since $J_\Psi>0$ almost everywhere in $\Omega$, then $\Psi$ has finite distortion.
\end{proof}

\begin{lemma}\label{lem5ec7d831}
	Let $\Psi:\Omega_1\to\Omega_2$ be a homeomorphism between open subsets of $\R^n$,
	$q>0$ and $p>n$.
	Suppose that 
	\[
	\Psi\in W^{1,p}_\loc(\Omega;\R^n), \quad\Psi^{-1}\in W^{1,p}_\loc(\Psi(\Omega);\R^n)
	\quad\text{and }
	J_\Psi(x)> 0\text{ for a.e.~}x\in\Omega .
	\]
	Then $\Psi$ is a homeomorphism of finite distortion and 
	\[
	K^\Psi_q\in \Lbs^{r}_\loc(\Omega) 
	\]
	with $r = \left( \frac{q}{p} + \frac{n}{p-n} \right)^{-1} $, which may be smaller than 1.
\end{lemma}
\begin{proof}
	The map $\Psi$ has finite distortion by Lemma~\ref{lem5ec7e29c}.
	Let $r>0$ and $U\Subset\Omega$.
	By Hölder inequality, we have
	\begin{equation}\label{eq5ea99bce}
	\int_U|K^\Psi_q|^r \dd x
	\le \left( \int_U \|D\Psi\|^{rq\alpha} \dd x \right)^{1/\alpha}
		\left( \int_U |J\Psi|^{-r\beta} \dd x \right)^{1/\beta} 
	\end{equation}
	for $\alpha,\beta\ge1$ with $\frac1\alpha+\frac1\beta=1$.
	If $rq\alpha=p$, then the first term of the right-hand side is finite.
	
	By Lemma~\ref{lem5e9a1230}, both $\Psi$ and $\Psi^{-1}$ are differentiable almost everywhere.
	Therefore, using Lemma~\ref{lem5e9a12a0}, it follows that, for almost every $x\in U$, 
	\[
	|J\Psi(x)|^{-1} = |J\Psi^{-1}(\Psi(x))| .
	\]
	Using the area inequality~\eqref{eq03161203}, we have for all $\gamma>1$
	\begin{align*}
	\int_U |J\Psi(x)|^{-\gamma} \dd x
	&= \int_U |J\Psi^{-1}(\Psi(x))|^{\gamma+1} |J\Psi(x)|\dd x \\
	&\le \int_{\Psi(U)} |J\Psi^{-1}(y)|^{\gamma+1} \,dy 
	\le \int_{\Psi(U)}  \|D\Psi^{-1}(y)\|^{n(\gamma+1)} \dd x ,
	\end{align*}
	where we used~\eqref{eq03161234} in the last step.
	Now, if $\gamma=r\beta$, then
	\[
	\int_U |J\Psi|^{-r\beta} \dd x
	\le \int_{\Psi(U)}  \|D\Psi^{-1}(y)\|^{n(r\beta+1)} \dd x ,
	\]
	which is finite when $n(r\beta+1)=p$.
	
	Finally, solving the three equations $\frac1\alpha+\frac1\beta=1$, $rq\alpha=p$ and $n(r\beta+1)=p$ in $\alpha,\beta$ and $r$, we get
	\[
	r = \left( \frac{q}{p} + \frac{n}{p-n} \right)^{-1} 
	\]
	and $\alpha = \frac{p}{rq} = 1+\frac{pn}{q(p-n)}>1$, by the assumption $p>n$.
	Therefore, we get from~\eqref{eq5ea99bce} that $\int_U|K_\Psi|^r \dd x<\infty$.
\end{proof}

\begin{proposition}[{\cite[Theorem~5.13 and equation~(5.18)]{MR3184742}}: Regularity of composition]\label{prop5ec7f3d8}
	Let $\Psi:\Omega_1\to\Omega_2$ be a homeomorphism between open subsets of $\R^n$
	and let $1\le\,p<\,q<\,\infty$.
	Define the composition operator $T_\Psi u=u\circ\Psi$ for $u:\Omega_2\to\R$.
	Suppose that $\Psi$ has finite distortion and
	\begin{equation}\label{assumpKq}
	K^\Psi_q\in \Lbs^{\frac{p}{q-p}}_\loc(\Omega_1) .
	\end{equation}
	Then 
	$T_\Psi$ is continuous from $W_\loc^{1,q}(\Omega_2)$ to $W_\loc^{1,p}(\Omega_1)$
	for every $1\le q\le\infty$,
	where the operator norm on $T_\Psi$ is controlled by $\|K^\Psi_q\|_{\Lbs^{\frac{p}{q-p}}}^{1/q}$ on balls.
\end{proposition}

\section{Preliminaries on flows of vector fields}\label{sec5edff28b}
In this section we present a part of the classical theory of flows of vector fields that we will need later.

\subsection{Well-posedness of vector fields}
We denote by $B(x,r)$ the Euclidean ball of radius $r$ and center $x$.
Recall that a vector field $b(t,x)$ is said to be \emph{autonomous} if it does not depend on $t$, 
\emph{non-autonomous} if it may depend on $t$.

\begin{definition}[Well posedness]\label{WP}
	Let $\Omega\subset\R^n$ open, $I\subset\R$ an open interval and $b:I\times\Omega\to\R^n$ a function.
	We say that $b$ is \emph{well-posed in $I\times\Omega$} if for every $(s,x)\in I\times\Omega$ there exist $\epsilon>0$ and, in the interval
	$(s-\epsilon,s+\epsilon)\subset I$, a unique absolutely continuous solution
	$\gamma:\,(s-\epsilon,s+\epsilon)\to\Omega$ of~\eqref{eq03161620}, that is, $\dot\gamma(t)=\,b(t,\gamma(t))$ for 
	a.e.~$t\in (s-\epsilon,s+\epsilon)$ and $\gamma(s)=x$.
Analogously, we say that $b$ is \emph{well-posed in $\overline{I}\times\Omega$} if for every $(s,x)\in\overline{I}\times\Omega$ 
	there exist $\epsilon>0$ and, in the interval $I\cap (s-\epsilon,s+\epsilon)$, a unique absolutely continuous solution
	$\gamma:\,I\cap (s-\epsilon,s+\epsilon)\to\Omega$ of~\eqref{eq03161620}.
\end{definition}

The following classical result provides a sufficient condition for well-posedness, in local and global form
with respect to the time variable.
The proof is rather classical and well known, see for instance~\cite{MR1028776,MR1929104}.

\begin{proposition}[Well-posedness with Osgood Condition]\label{prop628f93c2}
	Let $\omega:[0,+\infty)\to[0,+\infty)$ be non-decreasing with $\omega(0)=0$, $\omega(t)>0$ for $t>0$ and
	\begin{equation}\label{eq628f3982}
	 \int_0^1 \frac{1}{\omega(t)} \dd t = +\infty .
	\end{equation}
	
	Let $I\subset\R$ be an interval, 
	and $\Omega\subset\R^n$ an open set.
	Let $b:I\times\Omega\to\R^n$ be a function such that:
	\begin{enumerate}[label=(\alph*)]
	\item
	there exists a positive function $\phi\in L^1_\loc(I)$ such that 
		\begin{equation}\label{eq628f40a0}
		|b(t,x)-b(t,y)| \le \phi(t) \omega(|x-y|)
		\end{equation}
	for all $t\in I$ and $x,\,y\in\Omega$;
	\item for every $x\in\Omega$, the function $t\mapsto b(t,x)$ is measurable;
	\item there exists $m\in L^1_\loc(I)$ such that $|b(t,x)|\le m(t)$ for all $(t,x)\in I\times\Omega$.
	\end{enumerate}

	Then the following hold:
	\begin{enumerate}[label=(\roman*)]
	\item\label{prop628f93c2_wellpos}
	The vector field $b$ is well posed in $I\times\Omega$.
	For $(s,x)\in I\times\Omega$, denote by $\gamma_{(s,x)}$ the maximal solution of~\eqref{eq03161620} with $\gamma_{(s,x)}(s)=x$ and by $I_{(s,x)}$ its maximal domain of existence.
	Define also $\alpha(s,t) = \inf I_{(s,x)}$ and $\beta(s,t) = \sup I_{(s,x)}$.
	\item\label{prop628f93c2_boundary}
	If $\beta(s,x)\neq \sup I$, then $\lim_{t\to \beta(s,x)}\gamma_{(s,x)}$ exists and belongs to $\de\Omega$.
	Similarly for $\alpha(s,x)\neq \inf I$.
	\item\label{prop628f93c2_continuity}
	If $(s_j,x_j)\in I\times\Omega$ is a sequence converging to $(s_\infty,x_\infty)\in I\times\Omega$,
	with $\gamma_j=\gamma_{(s_j,x_j)}:I_{(s_j,x_j)}\to\Omega$ the corresponding maximal solutions, $j\in\N\cup\{\infty\}$,
	and if $s_\infty\in J\Subset I_{(s_\infty,x_\infty)}$, then $J\subset I_{(s_j,x_j)}$ for $j$ large enough and $\gamma_j$ converge to $\gamma_\infty$ uniformly on $J$.
	
	In particular, $\alpha$ is upper semicontinuous and $\beta$ is lower semicontinuous.
	\item\label{prop628f93c2_totheboundary}
	If $m\in L^1(I)$, then $\gamma_{(s,x)}$ is well defined also for $(s,x)\in\partial I\times\Omega$ and the limits
	\[
	\lim_{t\to\beta(s,x)} \gamma_{(s,x)}(t)
	\quad\text{and}\quad
	\lim_{t\to\alpha(s,x)} \gamma_{(s,x)}(t)
	\]
	exist in $\bar\Omega$.
	In particular, $b$ is well posed in $\overline I\times\Omega$.
	\end{enumerate}	
\end{proposition}

\subsection{The flow of a well-posed vector field}\label{sec5ed296d6}

Fix an open set $\Omega\subset\R^n$, an open interval $I\subset\R$ and a well-posed vector field $b:I\times\Omega\to\R^n$. 
For every $(s,x)\in I\times\Omega$, let $I_{(s,x)}=\,(\alpha(s,x),\beta(s,x))\subset I$ be the maximal open interval of existence of a solution $\gamma_{(s,x)}: I_{(s,x)}\to\Omega$ to the Cauchy problem~\eqref{eq03161620} with initial datum $\gamma_{(s,x)}(s)=x$.
Define 
\[
\Dom_b := \left\{ (t,s,x) : (s,x)\in I\times\Omega,\ t\in I_{(s,x)} \right\} \subset  I\times I\times\Omega ,
\]
and the \emph{flow} $X:\Dom_b\to\Omega$ of $b$ as $X(t,s,x) = \gamma_{(s,x)}(t)$.
For every $(t,s)\in I\times I$, we define $\Omega_{(t,s)} = \{x\in\Omega:(t,s,x)\in \Dom_b\}
= \{x\in\Omega: t\in I_{(s,x)}\}$ (which is possibly empty) and $X_{(t,s)}:\Omega_{(t,s)}\to\Omega$ as $X_{(t,s)}(x)=X(t,s,x)$.

An analogous notation can be used for the case of vector fields well posed in $\overline{I}\times\Omega$, just allowing $s$ and
$t$ to belong also to $\partial I$. 

We recall several well-known facts about flows of vector fields, 
referring to fix the ideas to the weaker case of well posedness in
$I\times\Omega$. See for instance~\cite{MR1929104} for reference.

\begin{remark}\label{rem62a45ceb}
	If $b:I\times \Omega \to\R^n$ 
	satisfies the hypothesis of Proposition~\ref{prop628f93c2},
	then the sets $\Dom_b$ and $\Omega_{(t,s)}$ are open
	and $X$ is continuous,
	thanks to Proposition~\ref{prop628f93c2}\ref{prop628f93c2_continuity}.
	
	Uniqueness gives the semigroup-type property
	\begin{equation}\label{eq5ec5219e}
	X(t_3,t_2,X(t_2,t_1,x)) = X(t_3,t_1,x)\ ,
	\end{equation}
	whenever the expressions make sense.
	Moreover, the map $X_{(t,s)}$ is a homeomorphism $\Omega_{(t,s)}\to\Omega_{(s,t)}$ with inverse~$X_{(s,t)}$.
	Finally, if $k\ge1$ and $b\in C^k(I\times\Omega;\R^n)$, then $X\in C^{k+1}(\Dom_b;\Omega)$.
\end{remark}

\begin{lemma}[{\cite[Theorem~3.1, Chap. V]{MR1929104}}]\label{lem04030932}
Any $b\in C^0(I;C^1(\Omega;\R^n))$ is well posed in $I\times\Omega$ and the homeomorphisms 
$X_{(t,s)}:\Omega_{(t,s)}\to\Omega_{(s,t)}$ are of class $C^1$.
	
Fix $(s,x)\in I\times\Omega$ and set the following functions defined on $I_{(s,x)}$
\begin{align*}
y(t)&:=\,D_xX(t,s,x) , \\
B(t)&:=\,(D_xb)(t,X(t,s,x)) , \\
J(t)&:=\det(D_xX(t,s,x)) , \\
\beta(t)&:=\div_x(b)(t,X(t,s,x)) = \trace(B(t)).
\end{align*}
Then $y$ and $J$ are the unique $C^1$ solutions of the following initial value problems:
\[
\begin{cases}
\dot{y}(t)& = B(t) y(t) , \\
y(s)&= \mathrm{Id} ,
\end{cases}
\qquad\text{and}\qquad
\begin{cases}
\dot{J}(t)&=\beta(t)J(t) , \\
J(s)&=1 .
\end{cases}
\]
\end{lemma}


When considering smooth approximations of vector fields, 
we need to control the convergence of the corresponding flows, as in Lemma~\ref{lem5e996a0e} below.
The following results are fairly well-known, but in absence of a reference precisely adapted to our purposes
we provide sketch of proofs for the reader's convenience. 
An analogous statement holds
for vector fields well posed in $\overline{I}$, considering the additional 
points $(s_\infty,x_\infty)\in\partial I\times\Omega$ and the stronger convergence 
in $L^1(I;C(\Omega';\R^n))$ for any $\Omega'\Subset\Omega$.

\begin{lemma}\label{lem5e996a0e}
	Let $\{b_h\}_{h\in\N\cup\{\infty\}}$ be well posed vector fields in $I$ and denote
	by $I^h_{(s,x)}$ their maximal existence times and by $X^h$ their flows.  
	If $b_h\to b_\infty$ in $L^1_\loc(I;C(\Omega';\R^n))$ for any $\Omega'\Subset\Omega$ 
	and $(s_h,x_h)\to (s_\infty,x_\infty)\in I\times\Omega$
	as $h\to \infty$, then for any compact interval $J\subset I^\infty_{(s_\infty,x_\infty)}$ one has
	$$
	J\subset I^h_{(s_h,x_h)}\qquad\text{for $h$ large enough}
	$$
	and $X^h(\cdot,s_h,x_h)$ converge uniformly to $X^\infty(\cdot,s_\infty,x_\infty)$ on $J$.
\end{lemma}
\begin{proof} 
Let us consider the compact curve $\Gamma=\{(v,X^\infty(v,s_\infty,x_\infty)):\ v\in J\}$
and let $U$ be the open set $J'\times\Omega'$, with $J\Subset J'\Subset I$ and $\Omega'\Subset\Omega$ chosen in such a way
that $\Gamma\subset U$. 
By assumption, 
 $\int_{J'}\sup_{\Omega'}|b_h(t,\cdot)-b(t,\cdot)|\dd t$ is infinitesimal as $h\to\infty$. Let us consider the maximal solutions $\gamma_h(v)$ for the ODE relative to
 $b_h$, starting from $x_h$ at $s_h$ and with graph remaining in $U$ (in particular, restrictions
 of $X^h(\cdot,s_h,x_h)$), and let $J_h\subset I^h_{(s_h,x_h)}$ their maximal existence
intervals. Since the curves $\gamma_h$ are 
 equicontinuous, it is clear that limit points of the graph of these curves
exist, and that any limit point is the graph of a solution $\gamma_\infty$ to the ODE relative to $b_\infty$, 
with $\gamma_\infty(s_\infty)=x_\infty$, defined on an interval $J_\infty$ and with the
 property that, as $v$ tends to one, if any, of the extreme points of $J_\infty$ different from $\partial I$, $(v,\gamma_\infty(v))$ tends to
 $\partial U$. Since $U\Subset I\times\Omega$, the well-posedness of $b_\infty$ yields that the curve $\gamma_\infty$ is
 not only a restriction of the maximal curve $X^\infty(v,s_\infty,x_\infty)$, $v\in I^\infty_{(s_\infty,x_\infty)}$, but also contains the
 curve $X^\infty(v,s_\infty,x_\infty)$, $v\in J$, since this restricted curve does not touch the boundary of $U$.
 \end{proof}
 
The previous lemma grants, in particular, the lower semicontinuity of $(s,x)\mapsto {\rm length}(I_{(s,x)})$ in $I\times\Omega$.
Moreover, a simple contradiction argument gives 
\begin{equation}\label{stronginclapprox}
\begin{split}
\text{ for any compact interval $J\subset I^\infty_{(s_\infty,x_\infty)}$ and any open set $A\Subset\Omega$, one has}\\
\exists\,\bar h\text{ such that }J\subset I^h_{(s,x)}\qquad\text{for each $h>\bar h$ and $(s,x)\in J\times A$}\,;
\end{split}
\end{equation}
\begin{equation}\label{eq:domains}
A\Subset\Omega^\infty_{(t,s)}\quad\text{implies}\quad
A\Subset\Omega^h_{(t,s)}\quad\text{for $h$ large enough}
\end{equation}
and the uniform convergence of $X_h(t,s,\cdot)$ to $X_\infty(t,s,\cdot)$ on $A$.

\section{Well-posedness with Orlicz condition}\label{sec5edff2b1}
In this section we are going to prove Theorem~\ref{thm5e999915}.
We fix some dimensional constants for $n\ge1$.
Let $\omega_n=|B(0,1)|$ be the volume of the unit Euclidean ball in $\R^n$ and $\sigma_{n-1}=n\omega_n$ the perimeter of $B(0,1)$;
let $\tau_n = |B(0,1)\cap B(q,1)|$ for any $q\in\de B(0,1)$;
finally, $C_n$ is the constant from Lemma~\ref{lem03211522} 
and $\kappa_n=2\tau_n^{-1}\omega_nC_n\sigma_{n-1}$ will appear in Lemma~\ref{lem03211603}.

\begin{lemma}\label{lem03211522} Let $b\in W^{1,1}(B(x,r);\R^n)$ and assume that $x$ is a Lebesgue point of $b$. Then,
	for some dimensional constant $C_n$, one has
	\begin{equation}\label{eq03201820}
	\ave_{B(x,r)} |b(x)-b(y)| \dd y 
	\le C_n \int_{B(x,r)} \frac{ \|D b(y)\| }{|x-y|^{n-1}} \dd y.
	\end{equation}
\end{lemma}
\begin{proof} We assume for simplicity $n\geq 2$.
	From the same argument of \cite[Lemma 1, Sec.4.5.2]{MR3409135}, based on a radial integration,
	 for all $x\in\R^n$, $r>0$, $\epsilon\in (0,1)$ and 
	$f\in C^1(B(x,r))$ we have
	\begin{equation}\label{eq03201820b1}
	\begin{split}
	\ave_{B(x,r)\setminus B(x,\epsilon r)} |f(x)-f(y)| \dd y 
	&\le\int_{\epsilon r}^r\ave_{B(x,s)} |y-x||D f(y)| \dd y \dd s\\
	&\leq
	\int_{\epsilon r}^rs\ave_{B(x,s)} |D f(y)| \dd y \dd s\, .
	\end{split}
	\end{equation}
	Assume now that $f\in L^1_{\rm loc}(\Omega)$, $x$ is a Lebesgue point of $f$, $B(x,r)\Subset\Omega$  and $f\in W^{1,1}(B(x,r))$.  Let $\delta:=\,{\rm dist}(B(x,r),\R^n\setminus\Omega) $,  $\Omega_\delta:=\,\{x\in\Omega:\,{\rm dist}(x,\R^n\setminus\Omega)>\,\delta/2\}$ and $(\rho_h)_h$ be a sequence of mollifiers in $\R^n$. Then it is well-defined the sequence of regularized functions $f_h:=\,f*\rho_h:\,\Omega_\delta\to\R$ if $h>\,1/\delta$ and it  satisfies the following properties: $f_h\in C^\infty(\Omega_\delta)$,
	\begin{equation}\label{fhconvfLP}
	f_h(z)\to f(z)\text{, for each $z\in \Omega_\delta$ Lebegue point of $f$}\text{, as }h\to\infty,
	\end{equation}
	\begin{equation}\label{fhconfW11}
	f_h\to f\text{ in }W^{1,1}(B(x,r))\,.
	\end{equation}
	Applying \eqref{eq03201820b1} with $f\equiv f_h$ and let $h\to\infty$, by \eqref{fhconvfLP} and \eqref{fhconfW11}, we get the same inequality when $f\in W^{1,1}(B(x,r))$ and $x$ is a Lebesgue point of $f$. Now
	 we can let $\epsilon\to 0$ and use Fubini's
	theorem in the right hand side to get
	\begin{equation}\label{eq03201820b}
	\ave_{B(x,r)} |f(x)-f(y)| \dd y 
	\le \frac{1}{n} \int_{B(x,r)} \frac{ |D f(y)| }{|x-y|^{n-1}} \dd y \, .
	\end{equation}
	
	Finally, notice that there is a constant $c_n$ such that, if $M$ is a $n\times n$ matrix with rows $M_j$, then $\sum_j|M_j|\le c_n\|M\|$.
	Therefore,
	\begin{align*}
	\ave_{B(x,r)} |b(x)-b(y)| \dd y 
	&\le \sum_{j=1}^n \ave_{B(x,r)} |b_j(x)-b_j(y)| \dd y \\
	&\le \frac 1n\sum_{j=1}^n \int_{B(x,r)} \frac{ |D b_j(y)| }{|x-y|^{n-1}} \dd y 
	\le \frac{c_n}{n}\ \int_{B(x,r)} \frac{ \|D b(y)\| }{|x-y|^{n-1}} \dd y ,
	\end{align*}
	which proves~\eqref{eq03201820} with $C_n=c_n/n$.
\end{proof}

Recall Jensen's inequality 
\begin{equation}\label{eq5ea98b05}
\Phi\left(\int_X\chi(x)\dd\mu(x) \right) \le \int_X \Phi(\chi(x)) \dd\mu(x) ,
\end{equation} 
where $\Phi:[0,+\infty)\to[0,+\infty)$ is a convex function, $\mu$ is a probability measure on $X$ and $\chi:X\to\R$ 
is a $\mu$-measurable non-negative function.
The following lemma is a direct application of Jensen's inequality.

\begin{lemma}
	Let $\Phi:[0,+\infty)\to[0,+\infty)$ be a convex function.
	If $n\ge1$, $x\in\R^n$, $r>0$, $\chi\in \Lbs^1_\loc(\R^n)$ is non-negative, then
	\begin{equation}\label{eq03202019}
	\Phi\left( \frac1r \int_{B(x,r)} \frac{\chi(z)}{|z-x|^{n-1}} \dd z \right)
	\le \frac{1}{r\sigma_{n-1}} \int_{B(x,r)} \frac{ \Phi(\sigma_{n-1}\chi(z)) }{ |x-z|^{n-1} } \dd z.
	\end{equation}
\end{lemma}
\begin{proof}
	Using polar coordinates, one easily checks that
	\[
	\int_{B(0,r)}\frac{1}{|z|^{n-1}} \dd z = r\sigma_{n-1}.
	\]
	Therefore, the measure $\dd\mu:=\frac{ \chi_{B(x,r)}(z) }{r\sigma_{n-1}} \frac{\dd z}{|x-z|^{n-1}}$ is a probability measure on $\R^n$ and Jensen's inequality~\eqref{eq5ea98b05} applies.
\end{proof}

\begin{lemma}\label{lem03211603}
Let $\Phi:[0,+\infty)\to[0,+\infty)$ be a convex non-decreasing function.
Let  $o\in\R^n$, $R>0$ and $b\in W^{1,1}(B(o,2R);\R^n)$.
Then, for all $\alpha\in(1,\frac{n}{n-1})$, and all $x,\,y\in B(o,R)$ distinct Lebesgue points of $b$, one has
\begin{equation}\label{eq03211539}
\begin{aligned}
\Phi\left( \frac{|b(x)-b(y)|}{|x-y|}  \right)
\le \frac{1}{\sigma_{n-1}} \left( \int_{B(o,2R)} \frac1{|z|^{\alpha(n-1)}} \dd z \right)^{1/\alpha}
\times \hspace{3cm} \\  \hspace{4cm} \times 
\frac1{|x-y|}
\left( \int_{B(o,2R)} \Phi\left( \kappa_n \|D b(z)\| \right)^{\frac{\alpha}{\alpha-1}} \dd z \right)^{\frac{\alpha-1}{\alpha}}.
\end{aligned}
\end{equation}
		
In the case $n=1$, the estimate~\eqref{eq03211539} is understood to hold for $\alpha\in(1,+\infty)$ and we also have
\begin{equation}\label{eq5eba52e8}
\Phi\left( \frac{|b(x)-b(y)|}{|x-y|}  \right)
\le 
\frac1{2|x-y|}
\int_{B(o,2R)} \Phi\left( \kappa_n \|D b(z)\| \right) \dd z .
\end{equation}
\end{lemma}
\begin{proof} Set $r:=|x-y|>0$ and $W=B(x,r)\cap B(y,r)$.
Notice that $|W|=\tau_nr^n$.
Then
\begin{align*}
|b(x)-b(y)| 
&\le \ave_W |b(x)-b(z)| \dd z + \ave_W |b(y)-b(z)| \dd z \\
&\le \frac{\omega_n}{\tau_n} \left(\ave_{B(x,r)} |b(x)-b(z)| \dd z + \ave_{B(y,r)} |b(y)-b(z)| \dd z \right) \\
&\le \frac{2\omega_n}{\tau_n} \sup_{w\in B(o,R)} \ave_{B(w,r)} |b(w)-b(z)| \dd z .
\end{align*}
	
Since $\Phi$ is non-decreasing, we have:
\begin{align*}
\Phi\left( \frac{|b(x)-b(y)|}{|x-y|}  \right)
&\le \sup_{w\in B(o,R)} \Phi\left( \frac{2\omega_n}{r\tau_n} \ave_{B(w,r)} |b(w)-b(z)| \dd z \right) \\
&\overset{\!\!\!\text{by }\eqref{eq03201820}}{\le}
\sup_{w\in B(o,R)} \Phi\left( \frac{2\omega_nC_n}{r\tau_n} \int_{B(w,r)} \frac{ \|D b(z)\| }{|w-z|^{n-1}} \dd z \right) \\
&\overset{\!\!\!\text{by }\eqref{eq03202019}}{\le}
\frac{1}{r\sigma_{n-1}} \sup_{w\in B(w,r)} \int_{B(w,r)} 
\Phi\left( \frac{2\omega_nC_n\sigma_{n-1}}{\tau_n}\,\|D b(z)\| \right) \frac1{|w-z|^{n-1}} \dd z \,.
	\end{align*}
	
If $n=1$, we have already obtained~\eqref{eq5eba52e8}.
However, if $\alpha>1$, we apply the Hölder inequality to obtain
\begin{multline*}
\int_{B(w,r)} \Phi\left( \kappa_n \|D b(z)\| \right) \frac1{|w-z|^{n-1}} \dd z \\
\le \left( \int_{B(w,r)} \Phi\left( \kappa_n \|D b(z)\| \right)^{\frac{\alpha}{\alpha-1}} \dd z \right)^{\frac{\alpha-1}{\alpha}}
\times \left( \int_{B(w,r)} \frac1{|w-z|^{\alpha(n-1)}} \dd z \right)^{1/\alpha} \\
\le \left( \int_{B(o,2R)} \Phi\left( \kappa_n \|D b(z)\| \right)^{\frac{\alpha}{\alpha-1}} \dd z \right)^{\frac{\alpha-1}{\alpha}}
\times \left( \int_{B(0,2R)} \frac1{|z|^{\alpha(n-1)}} \dd z \right)^{1/\alpha} .
\end{multline*}
Notice that $\int_{B(0,2R)} \frac1{|z|^{\alpha(n-1)}} \dd z < \infty$ if and only if $\alpha(n-1)<n$, i.e., $\alpha<\frac{n}{n-1}$ when $n>1$ or $\alpha<\infty$ when $n=1$. Applying this estimate to the former inequality, we obtain~\eqref{eq03211539}.
\end{proof}

\begin{proof}[Proof of Theorem~\ref{thm5e999915}]
Notice that  assumption~\ref{convTheta} implies that $\Theta$ is also convex, and condition~\eqref{condBI} is equivalent to
\begin{equation}\label{eq5ed11194}
\Theta^{-1}(s_1s_2) \le C_\Theta\, \Theta^{-1}(s_1) \Theta^{-1}(s_2)
\qquad\forall s_1,\,s_2\ge\Theta(C_\Theta) .
\end{equation}
Let $n>\,1$, fix $\alpha\in(1,\frac{n}{n-1})$ as in the assumption~\ref{convTheta}, 
so that $\Phi(s):=\Theta(s/\kappa_n)^{\frac{\alpha-1}{\alpha}}$ is convex, where $\kappa_n$ is the constant defined 
at the beginning of the section.
	
Fix $o\in\Omega$ and the corresponding $R>0$ and $c>0$ as in \eqref{eq5ea991ab}. We claim that 
the space continuous representative $\tilde{b}$ is well posed in
$I\times B(o,R)$.  Since $b:\,I\times B(o,R)\to\R^n$ is well posed and satisfies~\eqref{eq5ea991ab} if and only if $c\,b$ does, in the proof of the
claim we can assume with no loss of generality $c=1$.
Applying Lemma~\ref{lem03211603} 
with the above $\alpha\in(1,\frac{n}{n-1})$ and $\Phi$, we have, for almost every $t$ and  all $x,\,y\in B(o,R)$ distinct Lebegue points of 
$b(t,\cdot)$, one has
\[
\Theta\left( \frac{|b(t,x)-b(t,y)|}{\kappa_n |x-y|}  \right)^{\frac{\alpha-1}{\alpha}}
\le C(n,R,\alpha) \frac1{|x-y|}
\left( \int_{B(o,2R)} \Theta\left( \|D b(t,z)\| \right) \dd z \right)^{\frac{\alpha-1}{\alpha}} .
\]
We write $s_1\vee s_2$ for $\max\{s_1,s_2\}$ and $s_1\wedge s_2$ for $\min\{s_1,s_2\}$. Then, by applying~\eqref{eq5ed11194} 
twice we obtain that
\begin{multline*}
\frac{|b(t,x)-b(t,y)|}{\kappa_n|x-y|} \le C_\Theta^2 \Theta^{-1}(\Theta(C_\Theta)\vee C(n,R,\alpha)^{\frac{\alpha}{\alpha-1}})
\times\\\times
\Theta^{-1}\left(\Theta(C_\Theta)\vee  \int_{B(o,2R)} \Theta\left( \|D b(t,z)\| \right) \dd z \right) \times \\ \times
\Theta^{-1}\left(\Theta(C_\Theta)\vee \left(\frac1{|x-y|}\right)^{\frac{\alpha}{\alpha-1}}\right) .
\end{multline*}
If we set 
\begin{equation}\label{modcontomega}
\omega(\delta) := \delta\Theta^{-1}\left(\Theta(C_\Theta)\vee \left(\frac1{\delta}\right)^{\frac{\alpha}{\alpha-1}}\right) 
\end{equation}
and
\begin{equation}\label{functmodcontomega}
\varphi(t) := \kappa_n C_\Theta^2 \Theta^{-1}(\Theta(C_\Theta)\vee C(n,R,\alpha)^{\frac{\alpha}{\alpha-1}})\,\Theta^{-1}\left(\Theta(C_\Theta)\vee \int_{B(o,2R)} \Theta\left( \|D b(t,z)\| \right) \dd z \right)\, ,
\end{equation}
we obtain that $\omega\in C^0((0,\infty))$ and  $b$ satisfies
\begin{equation}\label{OLP}
|b(t,x)-b(t,y)|\le\,\varphi(t)\,\omega(|x-y|) 
\end{equation}
 for each $x,\,y\in B(o,R)$ Lebesgue points of $b(t,\cdot)$. From \eqref{OLP}, it follows that $b(t,\cdot)$ is uniformly continuous on the set of Lebesgue points of $b(t,\cdot)$ contained in $B(o,R)$. Since this set is dense in $B(o,R)$, it follows that there exists a unique continuous extension $\tilde{b}(t,\cdot):\,B(o,R)\to\R^n$  still satisfying \eqref{OLP} on the whole $B(o,R)$. Therefore we can conclude that $\tilde{b}(t,\cdot):\,\R^n\to\R^n$  is continuous, condition holds \eqref{OLP} on $B(o,R)$ and
 \begin{equation}\label{b=tildeb}
 \tilde b=\,b\text{ a.e.~on } I\times B(o,R)\,.
 \end{equation}
Moreover, it is clear that $\varphi\in \Lbs^1_\loc(I;\R)$, because $\Theta^{-1}$ is concave and the function in~\eqref{eq5ea991ab} belongs to  $\Lbs^1_\loc(I;\R)$ by assumption.
	
We claim that condition~\ref{item04021113} implies 
$\int_0^1\frac1{\omega(\delta)}\dd \delta = \infty$ and $\lim_{\delta\to0}\omega(\delta)=0$.
Indeed, on the one hand, using the monotonicity of $\Theta^{-1}$ and the change of variables $s=\Theta^{-1}((1/\delta)^{\frac{\alpha}{\alpha-1}})$, we obtain, if $\bar\delta:=\Theta(C_\Theta)^{\frac{1-\alpha}{\alpha}}\wedge 1$ and $\bar s=\Theta^{-1}((1/\bar\delta)^{\frac{\alpha}{\alpha-1}})$,
	\[
	\begin{split}
	\int_0^1\frac1{\omega(\delta)}\dd \delta
	\ge 
	\int_0^{\bar\delta} \frac{1}{\delta\Theta^{-1}((1/\delta)^{\frac{\alpha}{\alpha-1}})} \dd\delta
	= \frac{\alpha-1}{\alpha} \int_{\bar s}^\infty \frac{\Theta'(s)}{s\Theta(s)} \dd s 
	= +\infty.
	\end{split}
	\]
On the other hand, to prove ${\displaystyle\lim_{\delta\to0}}\omega(\delta)={\displaystyle\lim_{\delta\to0}} \delta\Theta^{-1}((1/\delta)^{\frac{\alpha}{\alpha-1}}) = 0$, we only need to show that the function $\delta\mapsto \delta\Theta^{-1}((1/\delta)^{\frac{\alpha}{\alpha-1}})$ is monotone for $\delta$ small enough, because the above integral is not bounded. Again with the change of variables $s=\Theta^{-1}((1/\delta)^{\frac{\alpha}{\alpha-1}})$, we see that this monotonicity is equivalent to the monotonicity of the function
$\displaystyle{\phi:s\mapsto s^{-1}\Theta(s)^{-\frac{\alpha-1}{\alpha}}}$ for $s$ large.
Inspecting the derivative of $\phi$, we see that
\[
\phi'(s) = \frac{ \Theta(s)^{-\frac{\alpha-1}{\alpha}} }{ s^2 } 
\left( -\frac{\alpha-1}{\alpha} \frac{\Theta'(s)}{s\Theta(s)} - 1 \right) < 0 ,
\]
and so $\phi$ is monotone for $s$ sufficiently large.  
We conclude that, for each $o\in\Omega$ and $t$ such that $\varphi(t)<+\infty$
 there exists $R>0$ such that $B(o,R)\subset\Omega$ and $\tilde{b}(t,\cdot):\,B(o,R)\to\R^n$ satisfying Osgood's condition 
\begin{equation}\label{OBoR}
|\tilde{b}(t,x)-\tilde{b}(t,y)|\le\,\varphi(t)\,\omega(|x-y|) \text{ for each }x,y\in B(o,R)\,.
\end{equation}

\begin{equation}\label{estimtildeb}
\begin{split}
|\tilde b(t,x)|&\le\, |\tilde{b}(t,x)-\tilde{b}(t,x_0)|+\,|\tilde{b}(t,x_0)|\le\,\varphi(t)\,\omega(|x-x_0|)+\,|\tilde{b}(t,x_0)|\\
&\le\,\omega(2R)\,\varphi(t)+|\tilde{b}(t,x_0)|\text{ for a.e.~}t\in I\text{ and for each }x,\,x_0\in B(o,R)\,.
\end{split}
\end{equation}
Now observe that, since $b\in L^1_{\rm loc}(I\times\Omega;\R^n)$, by \eqref{b=tildeb},  $\tilde b\in L^1_{\rm loc}(I\times B(o,R);\R^n)$.  Thus there exists $x_0\in B(o,R)$ such that $|\tilde b(t,x_0)|\in L^1_{\rm loc}(I)$. 
By \eqref{OBoR}, \eqref{estimtildeb} and Proposition~\ref{prop628f93c2} (i), we obtain that $\tilde{b}:\,I\times B(o,R)\to\R^n$ is well-posed.  
As a consequence also the vector field $\tilde{b}:\,I\times\Omega\to\R^n$  is well-posed.

Moreover,  if  \eqref{estimbuniftA} holds, it also follows that
\begin{equation*}
|\tilde b(t,x)|\le\,m(t)\text{ for a.e.~}t\in I, \text{for a.e.~}x\in B(o,R)\,.
\end{equation*}
Thus, by applying now Proposition~\ref{prop628f93c2} (iv), we can conclude that the vector field $\tilde{b}:\,\overline I\times B(o,R)\to\R^n$ is well-posed, which implies that  $\tilde{b}:\,\overline I\times \Omega\to\R^n$  is well-posed, too.

 When $n=1$, let $\Phi(s):=\Theta(s/\kappa_n)$. Then we can repeat the same arguments of the previous case,   showing that the space
 continuous representativre of $b$ satisfies Osgood's condition.
\end{proof}

The next proposition shows why we could not use a standard Sobolev embedding in the proof of Theorem~\ref{thm5e999915} to obtain a stronger Sobolev regularity for $b$.
See also Remark~\ref{rem628c97ad} below.

\begin{proposition}\label{prop627d11fd}
	Condition~\ref{item04021113} in Theorem~\ref{thm5e999915}
	implies that $\Theta$ cannot have polynomial growth,
	that is, 
	\begin{equation}\label{eq627e380c}
	\forall m\in\N
	\qquad
	\limsup_{s\to\infty} \frac{\Theta(s)}{s^m} =+\infty,
	\end{equation}
	but it does not imply that $\Theta$ has more than polynomial growth.
\end{proposition}
\begin{proof}
	On the one hand, we have
	\begin{align*}
	\infty 
	&= \int_1^\infty \frac{\Theta'(s)}{s\Theta(s)} \dd s 
	= \sum_{k=1}^\infty \int_k^{k+1} \frac{\Theta'(s)}{s\Theta(s)} \dd s \\
	&\le \sum_{k=1}^\infty \frac1k \int_k^{k+1} \frac{\Theta'(s)}{\Theta(s)} \dd s 
	= \sum_{k=1}^\infty \frac1k \left(\log(\Theta(k+1))-\log(\Theta(k))\right) \\
	&= -\log(\Theta(1)) + \sum_{k=2}^\infty \frac1{k^2-k} \log(\Theta(k)) 
		+ \limsup_{k\to\infty} \frac{\log(\Theta(k))}{k}
	\end{align*}
	If the series in the last row is infinite, then for every $0<\alpha<1$ there exists a sequence $k_j\to\infty$ such that
	\[
	\frac1{k_j^2-k_j} \log(\Theta(k_j)) \ge \frac1{k_j^{1+\alpha}} ,
	\]
	that is, for all $j$ with $k_j\ge2$,
	\[
	\Theta(k_j) \ge \exp\left( \frac12 k_j^{1-\alpha} \right) .
	\]
	If instead the series in the last row is finite, then $\limsup_{k\to\infty} \frac{\log(\Theta(k))}{k} = \infty$, that is, 
	there exists a sequence $k_j\to\infty$ such that
	\[
	\Theta(k_j) \ge \exp\left( k_j\right) .
	\]
	
	On the other hand, we cannot improve the $\limsup$ in~\eqref{eq627e380c} with a $\liminf$ (or a $\lim$). 
	Indeed,  with a similar estimate as above, we obtain
	\[
	\int_1^\infty \frac{\Theta'(s)}{s\Theta(s)} \dd s
	\ge - \frac12 \log(\Theta(1)) + \sum_{k=2}^\infty \frac{1}{k^2+k} \log(\Theta(k))  .
	\]
	Thus, taking a sequence $\{k_j\}_j$ sparse enough,
	one can construct a piece-wise linear increasing convex function $\Theta:[0,+\infty)\to[0,+\infty)$ such that $\Theta(k_j)=\exp(k_j^2+k_j)$ (thus $\int_1^\infty \frac{\Theta'(s)}{s\Theta(s)} \dd s=\infty$)
	and $\Theta(x_j)=x_j^2$ for some $k_j<x_j<k_{j+1}$ (hence 
	$\liminf_{s\to\infty} \frac{\Theta(s)}{s^2} < \infty$).

	Let us give more details about such construction.
	As a start, define $k_1=1$, $\Theta(1)=\exp(2)$ and $\alpha_1=\exp(2)$.
	Then define $\Theta(t)=\alpha_1t$ for $t\in[0,x_1]$, where $x_1$ is so that $\Theta(x_1)=x_1^2$, that is, $x_1=\alpha_1$.
	Iteratively, given $x_j$ and $\alpha_j$, and the function $\Theta$ defined in $[0,x_j]$ with $\Theta(x_j)=x_j^2$,
	define 
	\begin{align*}
	k_{j+1} &= \lceil 2x_j \rceil \\
	\alpha_{j+1} &= \frac{\exp(k_{j+1}^2+k_{j+1})-x_j^2}{k_{j+1}-x_j} \\
	\Theta(t) &= \Theta(x_j) + \alpha_{j+1} (t-x_j) \text{ for }t\in[x_j,x_{j+1}]\\
	&\text{ where $x_{j+1}$ is so that }
	\Theta(x_j) + \alpha_{j+1} (x_{j+1}-x_j)=x_{j+1}^2
	\end{align*}
	By construction, we have $\Theta(k_{j+1})=\exp(k_{j+1}^2+k_{j+1})$ and $\Theta(x_{j+1})=x_{j+1}^2$.
	Moreover, since $\exp(x^2+x)>x^2$ for $x\ge0$, we have $\alpha_{j+1}>\alpha_j$ and thus the resulting function $\Theta$ is convex;
	in fact, $\Theta$ is the sup of the linear functions $t\mapsto\Theta(x_j) + \alpha_{j+1} (t-x_j)$, $j=1,\dots$, and thus it is convex.
\end{proof}

\begin{remark}\label{rem628c97ad}
	Using the function $\Theta$ constructed in the previous Proposition~\ref{prop627d11fd},
	we can give an example of {\it a continuous vector field $b:\R\to\R$ that is well posed, 
	even though $b\notin W^{1,p}_\loc(\R)$ for all $p>2$}.
	
	Indeed, define $b(t) = \int_0^t \chi(s) \dd s$, where $\chi\in L^1(\R)$ is the function
	\[
	\chi(s) = \sum_{j=1}^\infty  \chi_{I_j}(s)\cdot x_j \text{ where } I_j := \left(\sum_{k=1}^{j-1} \frac{1}{k^2x_k^2} , \sum_{k=1}^{j} \frac{1}{k^2x_k^2}\right) .
	\]
	Notice that, by construction,
	\[
	\exp(4 x_j^2 + x_j)
	\le \exp(k_j^2+k_j) 
	=\Theta(k_{j+1})
	< \Theta(x_{j+1}) = x_{j+1}^2 ,
	\]
	and thus $x_{j+1} > \exp(2x_j^2+x_j) > \exp(x_j) > \exp(j)$ for all $j\ge1$.
	It follows that $\bigcup_{j=1}^\infty I_j$ is bounded and that
	\[
	\int_\R \chi(s)^p \dd s
	= \sum_{j=1}^\infty \frac{x_j^p}{j^2x_j^2}
	= \sum_{j=1}^\infty \frac{x_j^{p-2}}{j^2} 
	\]
	is finite if and only if $p\le 2$.	
\end{remark}

\subsection{A class of subexponential summability types}\label{sec5eb94e93}
Examples of functions $\Theta$ satisfying the properties 
listed in Theorem~\ref{thm5e999915}
are of the form $\FEL_{k,\beta}(s)$ as in~\eqref{sub2expofunct}, as we will show in this section.
It is clear that the subexponential summability 
of type $\FEL_{k,\beta}$ implies the subexponential summability of type $\FEL_{k',\beta'}$ for all $\beta\le\beta'\le 1$ and $k'\ge k$.

If $\left\|D_xb\right\|$ satisfies an exponential summability, then it is well-known that $b$ is well-posed. 
Indeed, in this case, the vector field $b$ satisfies a so-called Log-Lipschitz condition (see, for instance, \cite{MR4263701,MR1920427}).
Theorem~\ref{thm5e999915} extends the well-posedness to subexponential summability order $k\ge1$.
We will show in Section~\ref{sec5eb962e5} that the upper bound on $\beta$ for the subexponential summability order $1$  is in fact necessary; see also  \cite[Section 6]{MR3458193}, \cite[p. 1240]{MR3494396}.

Let us now show that $\FEL_{k,\beta}$ satisfies the assumptions of Theorem~\ref{thm5e999915}. Let us introduce some notation in order to better represent  $\FEL_{k,\beta}$. Let us denote by $E_k:\,\R\to\R$ the \emph{$k${\small th}-iterated exponential  function}, that is, by induction on $k$,
\[
E_1(s):= \exp(s),\quad E_{k+1}(s):= \exp\left(E_k(s)\right)\text{ if }s\in\R,\,k\ge\,1.
\]
Since $\lim_{s\to-\infty} E_1(s) = 0$, notice that $\lim_{s\to-\infty} E_k(s) = E_{k-1}(0)=:s_k$ for all $k>1$, and that $\lim_{k\to\infty}s_k=\infty$.
Denote by $L_k:E_k(\R)\to\R$ the \emph{$k${\small th}-iterated logarithm  function} as the inverse of $E_k$, that is
\[
L_k(s):=E_k^{-1}(s)\text{ if }s\in E_k(\R)=(s_k,+\infty)\,.
\]
Define then $P_k(s):=\prod_{j=1}^kL_j(s)$, so that one can easily check that
\[
	L_{k+1}'(s) = \frac{1}{sP_k(s)},
	\text{ so that }
	P'_k(s) = \frac{P_k(s)}{s} \sum_{j=1}^k \frac1{P_j(s)} .
\]
With this notation, we have 
\[
	\FEL_{k,\beta}(s) = \exp\left(\frac{s}{P_{k-1}(s)L_k^\beta(s)}\right) .
\]
A direct computation shows that
\begin{equation}\label{eq5fa556cf}
	\FEL_{k,\beta}' = \FEL_{k,\beta} H_{k,\beta}
	\quad\text{ where }\quad
	H_{k,\beta}
	= \frac{1}{P_{k-1} L_k^\beta} 
		\left( 1-\sum_{j=1}^{k-1}\frac1{P_j} - \frac{\beta}{P_k} \right) ,
\end{equation}
and that
\begin{equation}\label{eq5fa556ca}
 	H_{k,\beta}'
	= \frac{1}{sP_{k-1} L_k^\beta} 
		\bigg(
			- \Big(\sum_{j=1}^{k-1}\frac1{P_j}+\frac{\beta}{P_k}\Big)
				\Big(1-\sum_{j=1}^{k-1}\frac1{P_j}-\frac{\beta}{P_k}\Big) 
			+ \sum_{j=1}^{k-1}\frac1{P_j} \Big(\sum_{i=1}^j\frac1{P_i}+\frac{\beta}{P_k}\Big) 
		\bigg).
\end{equation}

\begin{proposition}\label{propFEL} 
	The function  $\FEL_{k,\beta}:\,\R\to\R$ satisfies the assumptions of Theorem~\ref{thm5e999915} for each integer $k\ge\,1$ and $0\le\,\beta\le\,1$.
\end{proposition}
\begin{proof}
{\bf Monotonicity of $\FEL_{k,\beta}$:}
From~\eqref{eq5fa556cf} it is evident that $\FEL_{k,\beta}'$ is positive for $s$ large enough and thus $\FEL_{k,\beta}$ is 
strictly increasing in $[\alpha,+\infty)$ for $\alpha$ large enough.

{\bf Verification of~\ref{convTheta}:} 
If $\gamma>0$, then
\[
\frac{\dd^2}{\dd s^2}\FEL_{k,\beta}(s)^\gamma
= \gamma\FEL_{k,\beta}(s)^\gamma \left(\gamma H_{k,\beta}^2 + H_{k,\beta}' \right) .
\]
Because of the presence of the factor $1/s$ in~\eqref{eq5fa556ca}, one can see that
$\lim_{s\to\infty}\frac{H_{k,\beta}'}{H_{k,\beta}^2}=0$ and thus $\frac{\dd^2}{\dd s^2}\FEL_{k,\beta}(s)^\gamma$ is positive for $s$
large enough.
Thus we obtain that $\FEL_{k,\beta}(s)^\gamma$ is convex for $s$ large.

{\bf Verification of~\ref{item04021111}:} 
Firstly, notice that $\log(st) = \log(s)+\log(t) \le\log(s)\log(t)$ for all $s,t>e$.
Therefore, by induction on $k$, we have $L_k(st)\le L_k(s)L_k(t)$ for all $s,t>E_n(1)$.
Hence, $P_k(st) = \prod_{j=1}^kL_j(st) \le \prod_{j=1}^kL_j(s)L_j(t) = P_k(s)P_k(t)$,
or all $s,t>E_n(1)$.

Secondly, condition~\eqref{condBI}, i.e., $\FEL_{k,\beta}(ts)\le\FEL_{k,\beta}(s)\FEL_{k,\beta}(t)$ for $s,t$ large enough, is equivalent to 
\[
\frac{s}{P_{k-1}(s)L_k(s)^\beta} + \frac{t}{P_{k-1}(t)L_k(t)^\beta}
	\le \frac{ st }{ P_{k-1}(st)L_k(st)^\beta } ,
\]
that is
\begin{equation}\label{eq5fa5649a}
\frac{P_{k-1}(t)L_k(t)^\beta}{t} + \frac{P_{k-1}(s)L_k(s)^\beta}{s}
	\le \frac{ P_{k-1}(s)L_k(s)^\beta P_{k-1}(t)L_k(t)^\beta}{P_{k-1}(st)L_k(st)^\beta} .
\end{equation}

Thirdly, on the one hand we have $\lim_{t\to\infty}\frac{P_{k-1}(t)L_k(t)^\beta}{t} = 0$, and thus the left-hand side of~\eqref{eq5fa5649a} smaller than 1 for $s$ and $t$ large.
On the other hand, by the initial observation,
\[
\frac{ P_{k-1}(s)L_k(s)^\beta P_{k-1}(t)L_k(t)^\beta}{P_{k-1}(st)L_k(st)^\beta} 
\ge
\frac{ P_{k-1}(s)L_k(s)^\beta P_{k-1}(t)L_k(t)^\beta}{P_{k-1}(s)P_{k-1}(t)L_k(s)^\beta L_k(t)^\beta} 
= 1 .
\]
So, inequality~\eqref{eq5fa5649a} holds true and so condition~\eqref{condBI}.

{\bf Verification of~\ref{item04021113}:}
Notice that, by~\eqref{eq5fa556cf}, for $s>E_k(1)$,
\[
\frac{ \FEL_{k,\beta}'(s) }{ s\FEL_{k,\beta}(s) } 
\sim_{s\to\infty} \frac{1}{sP_{k-1}(s)L_k(s)^\beta} .
\]
Thus, since $\int_{E_k(1)}^\infty \frac{1}{sP_{k-1}(s)L_k(s)^\beta} \dd s=\infty$ if and only if $\beta\le1$, we also have $\int_1^\infty\frac{ \FEL_{k,\beta}'(s) }{ s\FEL_{k,\beta}(s) } \dd s = +\infty$ if and only if $\beta\le 1$.
\end{proof}

\begin{remark}\label{sharpbehavsubexpo}
Observe that, if~\ref{item04021113} holds, then 
\begin{equation}\label{condTheta}
	\limsup_{s\to\infty}\frac{P_{k-1}(s)L_k(s)^{1+\alpha} \Theta'(s)}{\Theta(s)} = \infty
	\text{ for every  }\alpha>0\text{ and }k\ge1\,.
\end{equation}
\end{remark}

\section{Regularity of the flow with subexponential summability}\label{sec5edfede8}

This section is devoted to the proof of Theorem~\ref{thm5eb900db}
and its consequence in dimension 1 as written in Theorem~\ref{thm5ea9862b}.
Let us recall that, if $\Phi:\,[0,+\infty)\to [0,+\infty)$ is an increasing homeomorphism, so that $\Phi(0)=\,0$ and 
$\lim_{t\to +\infty}\Phi(t)=+\infty$, the {\it Orlicz space} $\Lbs^\Phi(\R^n)$ is the space of measurable functions $f:\R^n\to\R$ for which the \emph{Luxembourg norm}
\[
\|f\|_{\Lbs^\Phi}:=\,\inf\left\{\lambda>0:\,\int_{\R^n}\Phi\left(\frac{f(x)}{\lambda}\right)\dd x\le\,1\right\}
\]
is finite.  $L_\loc^\Phi(\R^n)$ will denote, as usual, the space of measurable functions $f$ such that $f\,\chi_K\in \Lbs^\Phi(\R^n)$ for each compact set $K\subset\R^n$.  If
\begin{equation}\label{exptoverlogt}
\Phi(t)=\exp\left(\frac{t}{\log^+t}\right)-1,
\end{equation}
then we will denote the obtained $\Lbs^\Phi(\R^n)$ (respectively  $\Lbs_\loc^\Phi(\R^n)$) by ${\rm Exp}\bigl({L}/{\log L}\bigr)$ (respectively ${\rm Exp}_\loc\bigl({L}/{\log L}\bigr)$).  When changing the reference measure  from the Lebegues measure to a Radon measure $\mu$ on $\R^n$, we will denote $\Lbs^\Phi(\R^n,\mu)$ (respectively $L_\loc^\Phi(\R^n,\mu)$) the associated Orlicz space and, if $\Phi$ is as in~\eqref{exptoverlogt}, by ${\rm Exp}_\mu\bigl({L}/{\log L}\bigr)$ (respectively ${\rm Exp}_{\mu, \loc}\bigl({L}/{\log L}\bigr)$).

\begin{proof}[Proof of Theorem~\ref{thm5eb900db}] 
We assume with no loss of generality that $b$ coincides
with its space continuous representative. 
It is easy to see that, from assumptions \eqref{globexsistb} and \eqref{eq5eb95263}, it follows that, for  a given $R>\,0$ and
\[
m(t):=\,(1+R\log^+R)\,\sup_{x\in B(0,R)}\frac{|b(t,x)|}{1+|x|\log^+|x|}\text{ for a.e.~}t\in I\,,
\]
then
\begin{equation}\label{estimbThmB}
m\in L^1(I)\text{ and } |b(t,x)|\le\, m(t)\text{ for a.e.~}t\in I, \text{ for each  } x\in B(0,R)
\end{equation}
and
\begin{equation}\label{eq5eb95263vecchia}
\exp\left( \frac{ \|D_xb\| }{ \log^+\|D_xb\| } \right)\in L^1(I;L^1_{\rm loc}(\R^n))\,.
\end{equation}

Thus, by \eqref{estimbThmB} and \eqref{eq5eb95263vecchia}, the vector field $b$ is well-posed in $\overline{I}\times\R^n$ by Theorem~\ref{thm5e999915} and Proposition~\ref{propFEL}. Moreover the flow $X$ associated to $b$ is globally defined thanks
to the  growth condition \eqref{globexsistb}, so that $X\in C^0(\overline I\times\overline I\times\R^n;\R^n)$ 
(see, for instance, \cite[Theorem~5.1, Chap. III]{MR1929104}).
 In particular, $b$ satisfies the assumptions of  \cite[Main Theorem]{MR3494396}. Therefore
 the push-forward measure $X(t,s,\cdot)_\# \gamma_n$ is absolutely continuous with respect to the Gaussian measure
 $\gamma_n$ with density $w(t,s)$ belonging to $\Lbs^{\Phi_\alpha}(\R^n,\gamma_n)$ for each $0<\alpha<\alpha_0(s,t)$, where 
\[
\alpha_0(t,s) = \exp\left(-16\, e^2 \int_s^t\|\div_{\gamma_n} b(v,\cdot)\|_{{\rm Exp}_{\gamma_n}\left(\frac{L}{\log L}\right)} \dd v \right) ,
\]
where
\[
\div_{\gamma_n} b(v) = \div b(v) - v\cdot b(v) ,
\]
that is, $\div_{\gamma_n}$ is the adjoint of the gradient operator with respect to the measure $\gamma_n$. 
Notice that $\alpha_0$ is a continuous function with $\alpha_0(t,t)=1$ for every $t\in I$.
	
	To complete the proof, it is enough to show that
	\begin{equation}\label{RNderX0Ln}
	\frac{\dd}{\dd \Lbm^n}(X (t,s,\cdot)_\# \Lbm^n)(x) = D(x) w(t,s)(x) \quad\text{for a.e.~} x\in\R^n\,,
	\end{equation}
	where 
	 $D(x)=\exp\left({\frac{|X(s,t,x)|^2-|x|^2}{2}}\right)$
	and ${\frac{d}{d\Lbm^n}}$ denotes the Radon-Nikodym derivative with respect to $\Lbm^n$.
	Indeed, since the map $D:\R^n\to [0,\infty)$ is continuous and since $\Lbs^{\Phi_\alpha}(\R^n,\mu)\subset \Lbs^{\Phi_\alpha}_\loc(\R^n)$, 
	we can conclude~\eqref{eq5eb91e42}.
	
	In order to prove~\eqref{RNderX0Ln}, fix $t,\,s\in I$ and set $\psi:= X (s,t,\cdot)$.
	By the differentiation theorem for Radon measures (see, for instance, \cite{MR1857292}), we have
	\[
	\frac{\dd}{\dd \Lbm^n} \psi_\# \Lbm^n = \lim_{r\to0^+} \frac{\Lbm^n(\psi(B(x,r))) }{ \Lbm^n(B(x,r)) },
	\]
	for almost every $x\in\R^n$.
	Notice that
	\[
	\frac{ \Lbm^n(\psi(B(x,r)))}{ \Lbm^n(B(x,r))}
	= \frac{ \Lbm^n(\psi(B(x,r)))}{\gamma_n(\psi(B(x,r)))}
		\frac{\gamma_n(\psi(B(x,r)))}{\gamma_n(B(x,r))}
		\frac{\gamma_n(B(x,r)))}{ \Lbm^n(B(x,r))} ,
	\]
	where, again by differentiation theorem for Radon measures,
	\begin{align*}
	\lim_{r\to 0^+} \frac{\gamma_n(B(x,r)))}{ \Lbm^n(B(x,r))}
		&= \frac{1}{(2\pi)^{n/2}} \exp\left(-\frac{|x|^2}{2}\right) , \\
	\lim_{r\to 0^+}\frac{\gamma_n(\psi(B(x,r)))}{\gamma_n(B(x,r))}
		&= w(t,s)(x) , \\
	\lim_{r\to 0^+} \frac{ \gamma_n(\psi(B(x,r))) }{ \Lbm^n(\psi(B(x,r))) } 
		&=  \frac{1}{(2\pi)^{n/2}} \exp\left(- \frac{|\psi(x)|^2}{2}\right) .
	\end{align*}
	A short computation leads to~\eqref{RNderX0Ln}.
\end{proof}

When the spatial dimension $n$ equals 1, we have stronger results. 

\begin{proof}[Proof of Theorem~\ref{thm5ea9862b}]
	By Theorem~\ref{thm5eb900db}, $b$ is well-posed and both maps $X_{(t,s)}:\Omega_{(t,s)}\to\Omega_{(s,t)}$ and $X_{(t,s)}^{-1} = X_{(s,t)}:\Omega_{(s,t)}\to\Omega_{(t,s)}$ satisfy the Lusin (N) condition.
	Thus, by a well-known result of real analysis (see, for instance, \cite[Theorem 7.45]{gariepy1995modern}) they must be locally absolutely continuous.
\end{proof}

\begin{remark}\label{rem5ede0b6b}
We do not know whether, if $n\ge2$, there exists a flow $X$ associated to a vector field $b$ satisfying~\eqref{eq5eb95263}, but $X(t,s,\cdot)\notin W^{1,1}_\loc(\R^n ;\R^n)$ for some $t,s\in\R$.
Recall it has been proven in \cite{MR3634023} that there exists an almost everywhere approximately differentiable, orientation and measure preserving  homeomorphism $\Psi:\R^n\to\R^n$, whose Jacobian is equal to $-1$ almost everywhere on the unit $n$-dimensional cube $Q$, $\Psi(x)=x$ if $x\in \R^n\setminus Q$  and  $\Psi\notin W^{1,1}(Q;\R^n)$. 
In particular it holds that
\[
\Psi_\#\Lbm^n\ll\Lbm^n\text{ and }\frac{\dd}{\dd\Lbm^n}(\Psi_\#\Lbm^n)(x)=1
\text{ for a.e.~} x\in\R^n .
\]
We are not aware whether such a homeomorphism $\Psi$ could be induced by the  flow of a suitable vector field.
\end{remark}

\begin{remark} 
Sobolev regularity stated in Theorem~\ref{thm5ea9862b} 
is sharp, as we show in Example~\ref{sec5eb962e5} below. 
Moreover it can be compared with
 the result in \cite{MR4061966}, which proves existence and uniqueness of a unique flow $X$, with $D_xX(t,s,\cdot)\in A_\infty(\R)$, provided  that $D_xb\in L^1((0,T);BMO(\R))$ (see \cite[Chap. 7]{MR3243741} for the definition of  the class of functions $BMO(\R^n)$).
Notice that a function in $BMO$ admits exponential summability, 
as defined in Section~\ref{sec5eb94e93}.
In particular,~\eqref{eq5eb95263} holds, too. 
Whereas there are vector fields $b$ with $D_xb$ exponentially summable, but $D_xb\notin\Lbs^1((0,T);BMO(\R))$ (see Example~\ref{sec5ec2a1ee} below).
Notice also that there are flows associated to well-posed vector fields, which are not absolutely continuous (see Example~\ref{sec5ec2a20e} below).
\end{remark}

\section{Regularity of the flow with exponential summability}\label{sec5edff316}

This section is devoted to the proof of Theorem~\ref{thm5e9a10ec}. Under the stronger regularity assumption
\begin{equation}\label{eq:extrareg}
b\in C^0(I;C^1(\Omega;\R^n))
\end{equation}
we are able to derive a more precise a priori estimate, involving the quantity $\ell(s,x)$, namely the length of the 
maximal interval $I_{(s,x)}$.
Notice that we clearly have $\ell(s,x)\leq\ell$,
 with $\ell$ the length of $I$.
Moreover, when $b$ is a well posed vector field, the maximal integral curves of $b$ stop either at the boundary of $I$ or of $\Omega$ and thus, for all $s\in I$,
\begin{equation}\label{eq6283574e}
\ell(s,x)\geq \min\{\ell,{\rm dist}(x,\partial\Omega)/\sup |b|\} .
\end{equation}
From this follows that,
$\Lambda_p'\le\Lambda_p$, where $\Lambda_p'$ is defined in~\eqref{eq63eb586e} below and $\Lambda_p$ in~\eqref{eq:geometric}.
In particular, 
 the finiteness condition~\eqref{eq:defLambda_p_bis} implies the finiteness of $\Lambda'_p$ as stated in~\eqref{eq63eb586e}.

\begin{theorem}\label{thm5e9a10ec_bis}
	Let $I\subset\R$ and $\Omega\subset\R^n$ be bounded open sets and let $b\in C^0(I;C^1(\Omega;\R^n))$ be a bounded vector field. Assume
	 that for some $p> 2n$ one has
	\begin{equation}\label{eq:defLambda_p_bis}
	\int_I\int_\Omega
				{\rm dist}(x,\de\Omega)^{\frac{n}{n-p}}\exp\left( \frac{\ell p^2}{p-n}  \|D_xb(s,x)\| \right) 
				\dd x\dd s 
	<+\infty.
	\end{equation}
	Then $b$ is a well-posed vector field in $\overline{I}\times\Omega$ thanks to Theorem~\ref{thm5e999915}.
	
	For all $t\in \overline{I}$ 
	\begin{equation}\label{eq5ff85eabbis}
	\text{for almost every $s\in{I}$,\ }
		X(t,s,\cdot)\in W^{1,p}(\Omega_{(t,s)};\R^n)\quad\text{and}\quad
		X(s,t,\cdot)\in W^{1,p}(\Omega_{(s,t)};\R^n) .
	\end{equation}
	
	
	Moreover, we have that, for all $t\in \overline{I}$,
	\begin{equation}\label{eq63ea3e58}
		\int_I\int_{\Omega_{(t,s)}} \|D_xX(t,s,x)\|^p \dd x\dd s 
		\leq 
			\Lambda'_p ,
	\end{equation}
	where
	\begin{equation}\label{eq63eb586e}
	\Lambda_p':=
	\int_I\int_\Omega \left(\frac{\ell}{\ell(s,x)}\right)^{\frac{n}{p-n}} \exp\biggl(\frac{\ell p^2}{p-n}\|D_xb(s,x)\|\biggr)\dd x \dd s 
	< \infty.
	\end{equation}
	
	If we also assume that there exists $\hat\ell>0$ such that 
	\begin{equation}\label{eq63eb4e54}
		\ell(t,x)\ge\hat\ell 
		\text{ for all $t\in I$ and $x\in\Omega$},
	\end{equation}
	then, for every $s,t\in I$, 
	\begin{equation}\label{eq63ea3f57}
		\int_{\Omega_{(t,s)}} \|D_xX(t,s,x)\|^p \dd x
		\leq
		\frac1{\hat\ell} \left(\frac{\ell}{\hat\ell}\right)^{\frac{n^2}{p(p-n)}} \Lambda''_p
	\end{equation}
	where
	\[
		\Lambda''_p 
		:= \int_{I}\int_\Omega 
			\exp\left( \frac{\ell p^2}{p-n}\|D_xb(v,y)\| \right) 
		\dd y\dd v .
	\]
	Moreover, if there exists an open set $\Omega'\Subset\Omega$ such that
	\begin{equation}\label{cptsuppb}
			\spt(b(t,\cdot))\subset \Omega'\text{ for each }t\in I\,,
	\end{equation}
	then~\eqref{eq:defLambda_p_bis} implies~\eqref{eq63ea3f57} also for $p>n$.	
\end{theorem}

\begin{remark}\label{rem63eb57a1}
	Typical cases where where~\eqref{eq63eb4e54} holds are the following.
	First, as we already observed in the statement above, 
	if there exists an open set $\Omega'\Subset\Omega$ such that~\eqref{cptsuppb} holds
	then it turns out that 
	$\Omega_{(t,s)}=\,\Omega$ and $\ell(t,x)=\ell$
	for every $t,\,s\in \overline{I}$ and $x\in\Omega$.
		
	Second, if $\de\Omega$ is smooth and compact and $b$ is tangent to $\de\Omega$, 
	then we have again $\Omega_{(t,s)}=\,\Omega$ and $\ell(t,x)=\ell$
	for every $t,\,s\in \overline{I}$ and $x\in\Omega$.
	
	Third, one can easily have $\hat\ell<\ell$:
	for example, if on the plane $\R^2$ with coordinates $(x,y)$, we take $b(t,(x,y)) = \de_x$ and $\Omega = \{(x,y):-2+y < x < 2-y,\ 0<y<1\}$.
	In this case we have $\ell(t,(x,y))=4-2y$.
\end{remark}

The proof of Theorem~\ref{thm5e9a10ec_bis} is based on the well-known estimate~\eqref{eq09071038} below and a bootstrap argument in three steps.
The proof gives also an intermediate estimate~\eqref{eq63c7c112} on subdomains $A\Subset\Omega$ that involves the function $\ell(s,x)$ in place of ${\rm dist}(x,\de\Omega)$.

\begin{remark}\label{rem:torus} The proof of Sobolev regularity of the map $X(t,s,\cdot)$ becomes even simpler when the
spatial domain $\Omega$ is replaced by a compact Riemannian manifold $M$ without boundary, as for instance the 
$n$-dimensional torus $\bb T^n$ considered in \cite{MR4263701} (see also 
\cite{MR3912727,MR3494396,MR3458193,MR3906270,MR3778580}).
Indeed, in this case the quantity $\ell(s,x)$ equals
the length of $I$. The extension to the case of Sobolev regularity with respect to the space variable,
along the lines of Theorem~\ref{thm5e9a10ec} (i.e., dropping assumption \eqref{eq:extrareg}), simply requires a global approximation
of $b$ by more regular vector fields. In the case $M=\bb T^n$, one may argue by convolution with respect to the
space variable, viewing the vector field as a spatially periodic one defined in $I\times\R^n$.
\end{remark}

The first lemma is one of the many variants of  Gronwall's lemma, whose simple proof is omitted.

\begin{lemma}\label{lem04030936}
	Let $f:[s,t]\to [0,+\infty)$ be an absolutely continuous function and $\beta\in L^1(s,t)$ non-negative.
	If $f'\le \beta f$ a.e.~in $(s,t)$ and $f(s)=1$, then $f(t)\le \exp(\int_s^t\beta(v)\dd v)$.
\end{lemma}

\begin{lemma}\label{lem04031001}
	Let $b\in C^0(I;C^1(\Omega;\R^n))$ and let $X:\Dom_b\to\Omega$ the corresponding flow.
	Then, for every $(t,s,x)\in\Dom_b$ one has
	\begin{equation}\label{eq09071038}
	\|D_xX(t,s,x)\|
	\le  \exp\left| \int_s^t \|D_xb(v,X(v,s,x))\| \dd v \right| .
	\end{equation}
\end{lemma}
\begin{proof}
	Set $y(v)=D_xX(v,s,x)$, $B(v)=(D_xb)(v,X(v,s,x))$ and assume, to fix the ideas, $s\le t$ (the proof in the other case is 
	similar). Since, thanks to Lemma~\ref{lem04030932}, $y$ is of class $C^1$, the function 
	$[s,t]\ni  v\mapsto|y(v)|$ is  absolutely continuous.
	Using once more Lemma~\ref{lem04030932}, we have
	\[
	\frac{\dd}{\dd v} |y(v)| 
	= \langle \frac{y(v)}{|y(v)|},\dot y(v) \rangle
	= \langle \frac{y(v)}{|y(v)|},B(v) y(v) \rangle
	\le \|B(v)\| |y(v)|\text{ for a.e.~}v\in (s,t)\, .
	\]
	Lemma~\ref{lem04030936} implies~\eqref{eq09071038}, where
	the absolute value in the argument of the exponential is necessary when $t<s$.
\end{proof}

\begin{proof}[Proof of Theorem~\ref{thm5e9a10ec_bis}]  
Fix $t\in I$ and $A\subset\Omega$ open;
 first, we use
	Lemma~\ref{lem04031001} 
	and Jensen's inequality~\eqref{eq5ea98b05} to get
\begin{equation}\label{eq63c6ed39}
\begin{aligned}
	\int_I\int_{A_{(t,s)}} \|D_xX(t,s,x)\|^p \dd x\dd s
	&\le \int_I\int_{A_{(t,s)}}\exp\left( p \left|\int_s^t \|D_xb(v,X(v,s,x))\| \dd v\right| \right)
	\dd x \dd s\\
	&\le \int_I\int_{A} \int_{I} \frac{\chi_{A(v,s)}(x)}{\ell_A(s,x)} \exp\left( \ell p\|D_xb(v,X(v,s,x))\|\right)
	\dd v\dd x\dd s, 
\end{aligned}
\end{equation}
where we write $\ell_A(s,x)$ for the length of the  interval of the maximal integral curve of $b$ in $A$ that starts from $x$ at time $s$.
Setting $J(s,v,\cdot)=J_X(s,v,\cdot)=\det D_x X(s,v,\cdot)$,
we apply 
the change of variable $y=X(v,s,x)$, i.e., $X(s,v,y)=x$
and $\chi_{A(v,s)}(x)\dd x = J(s,v,y) \chi_{A(s,v)}(y) \dd y$;
notice in particular that $x\in A(v,s)$ if and only if $y=X(v,s,x)\in A(s,v)$.
Together with Hadamard's inequality~\eqref{eq03161234} 
and the Hölder inequality with exponents $q=p/n$ and $q'=p/(p-n)$,
we get
\begin{equation}\label{eq63eb3f4f}
\begin{aligned}
	&\int_I\int_{I}\int_{A} \frac{\chi_{A(v,s)}(x)}{\ell_A(s,x)} \exp\left( \ell p\|D_xb(v,X(v,s,x))\|\right)
	\dd x\dd v\dd s \\
	&\leq \int_I\int_{I}\int_{A} 
	\frac{ \chi_{A(s,v)}(y) }{ \ell_A(s,X(s,v,y)) } \exp\left( \ell p\|D_xb(v,y)\| \right) \|D_y X(s,v,y)\|^{n}
	\dd y\dd v\dd s \\
	&\leq
	\left( 
		\int_I\int_{I}\int_{A} 
		\frac{ \chi_{A(s,v)}(y) }{ \ell_A(s,X(s,v,y))^{q'} } \exp\left( \ell pq'\|D_xb(v,y)\| \right) 
	\dd y\dd v\dd s
	\right)^{1/q'}
	\times \\ &\phantom{AAAAAAAAAAAAAAAAAAAAA}\times
	\left(
	\int_I\int_{I}\int_{A} 
	 \|D_y X(s,v,y)\|^{nq}
	\dd y\dd v\dd s
	\right)^{1/q} .
\end{aligned}
\end{equation}
The identity $\ell_A(s,X(s,v,y)) = \ell_A(v,y)$ gives us
\[
	\int_I \frac{ \chi_{A(s,v)}(y) }{ \ell_A(s,X(s,v,y))^{q'} } \dd s
	= \ell_{A(v,y)}(y)^{1-q'} ,
\]
and we conclude that
\begin{equation}\label{eq63c6ec74}
\begin{aligned}
	&\int_I\int_{A_{(t,s)}} \|D_xX(t,s,x)\|^p \dd x\dd s \\
	&\leq
	\left( 
		\int_{I}\int_{A} 
		\ell_A^{1-q'}(v,y) \exp\left( \ell pq'\|D_xb(v,y)\| \right) 
	\dd y\dd v
	\right)^{1/q'}
	\times \\ &\phantom{AAAAAAAAAAAAAAAAAAAAA}\times
	\left(
	\int_I\int_{I}\int_{A} 
	 \|D_y X(s,v,y)\|^{nq}
	\dd y\dd v\dd s
	\right)^{1/q} .
\end{aligned}
\end{equation}

If we take $A\Subset\Omega$, the latter triple integral is finite.
By integrating this inequality with respect to $t$,
and then rearranging the terms,
 we obtain that
\begin{equation}\label{eq63c7ad49}
\int_I\int_I\int_{A_{(t,s)}}\|D_xX(t,s,x)\|^p\dd x\dd s \dd t
\le \ell^{q'} \int_I\int_A
			\ell_A^{1-q'}(v,y)\exp\left( \ell p q' \|D_xb(v,y)\| \right) 
			\dd y\dd v
\end{equation}
If we plug the estimate~\eqref{eq63c7ad49} into~\eqref{eq63c6ec74},
we obtain
\begin{equation}\label{eq63c7c112}
	\int_I\int_{A_{(t,s)}} \|D_xX(t,s,x)\|^p \dd x\dd s
	\leq \ell^{\frac{n}{p-n}} 
	\int_I\int_A
			\ell_A^{\frac{n}{n-p}}(v,y)\exp\left( \frac{\ell p^2}{p-n} \|D_xb(v,y)\| \right) 
			\dd y\dd v .
\end{equation}

To prove~\eqref{eq63ea3e58},
we need to take a limit in~\eqref{eq63c7c112} along a sequence of sets $A\Subset\Omega$ that fills $\Omega$.
To this aim, we define
\[
A^\epsilon := \{ x\in\Omega: {\rm dist}(x,\de\Omega) > \epsilon \}
\]
for $\epsilon>0$, plug $A^\epsilon$ in~\eqref{eq63c7c112} and take $\epsilon\to0$.

For the left-hand side of~\eqref{eq63c7c112}, we observe that
 $A^{\epsilon_1}_{(t,s)} \subset A^{\epsilon_2}_{(t,s)}$ whenever $\epsilon_1>\epsilon_2$. 
 Thus, we have 
\[
\lim_{\epsilon\to0}\int_I\int_{A^\epsilon_{(t,s)}}\|D_xX(t,s,x)\|^p\dd x\dd s = \int_I\int_{\Omega_{(t,s)}}\|D_xX(t,s,x)\|^p\dd x\dd s .
\]

The right-hand side of~\eqref{eq63c7c112} is more tricky.
Define
$f(v,y) := \exp\left( \ell p q' \|D_xb(v,y)\| \right)$,
 $\delta(x) := {\rm dist}(x,\de\Omega)$,
 and $\alpha = \frac{n}{n-p}$.
Notice that, since $p\geq 2n$, we have $\alpha\in(-1,0)$.

We claim that
\begin{equation}\label{eq63c7c610}
	\liminf_{\epsilon\to0} \int_{A^\epsilon} \int_I \ell_{A^\epsilon}^{\alpha}(v,y) f(v,y) \dd v\dd y
	\leq \int_{\Omega}\int_I \ell_\Omega^{\alpha}(v,y) f(v,y) \dd v\dd y .
\end{equation}
If $K\Subset\Omega$, the monotone convergence theorem implies that 
\[
\lim_{\epsilon\to0} \int_{K}\int_I \ell_{A^\epsilon}^{\alpha}(v,y) f(v,y) \dd v\dd y
	= \int_{K}\int_I \ell_\Omega^{\alpha}(v,y) f(v,y) \dd v\dd y .
\]
Thus,~\eqref{eq63c7c610} is shown if we can prove that
\begin{equation}\label{eq63c7c693}
	\inf_{K\Subset\Omega} 
	\liminf_{\epsilon\to0}
	\int_{A^\epsilon\setminus K} \int_I \ell_{A^\epsilon}^{\alpha}(v,y) f(v,y) \dd v\dd y
	= 0 .
\end{equation}
Indeed, thanks to~\eqref{eq63eb586e}, for every $K\Subset\Omega$ we have
\begin{multline*}
	\left| \liminf_{\epsilon\to0} \int_{A^\epsilon} \int_I \ell_{A^\epsilon}^{\alpha}(v,y) f(v,y) \dd v\dd y
	- \int_{\Omega}\int_I \ell_\Omega^{\alpha}(v,y) f(v,y) \dd v\dd y \right| \\
	\le
	\liminf_{\epsilon\to0}
	\int_{A^\epsilon\setminus K} \int_I \ell_{A^\epsilon}^{\alpha}(v,y) f(v,y) \dd v\dd y
	+ \int_{\Omega\setminus K}\int_I \ell_\Omega^{\alpha}(v,y) f(v,y) \dd v\dd y .
\end{multline*}
If there is a sequence $K_j\Subset\Omega$ such that 
\[
\lim_{j\to\infty}
\liminf_{\epsilon\to0}
	\int_{A^\epsilon\setminus K_j} \int_I \ell_{A^\epsilon}^{\alpha}(v,y) f(v,y) \dd v\dd y
= 0 ,
\]
then $\lim_{j\to\infty} |\Omega\setminus K_j| = 0$, and 
$\lim_{j\to\infty} \int_{\Omega\setminus K_j}\int_I \ell_\Omega^{\alpha}(v,y) f(v,y) \dd v\dd y  = 0$,
and so~\eqref{eq63c7c610} holds.

To prove~\eqref{eq63c7c693} we use the inequality~\eqref{eq6283574e}.
If $K$ is large enough,
so that ${\rm dist}(y,\de A^{\epsilon}) < \ell$ for all $y\in A_\epsilon\setminus K$,
the inequality~\eqref{eq6283574e} simplifies to
\[
\ell_{A^\epsilon}(v,y) \|b\|_{L^\infty} \geq {\rm dist}(y,\de A^{\epsilon}).
\]
We use the latter inequality
to compute the following averaged integral
\begin{align*}
\frac1\eta \int_0^\eta 
	&\int_{A^\epsilon\setminus K} \int_I \ell_{A^\epsilon}^{\alpha}(v,y) f(v,y) \dd v \dd y
	\dd\epsilon \\
&\leq \frac1\eta \int_0^\eta 
	\int_{A^\epsilon\setminus K} \int_I \frac{ {\rm dist}(y,\de A^\epsilon)^{\alpha} }{ \|b\|_{L^\infty}^{\alpha} } f(v,y) \dd v \dd y
	\dd\epsilon \\
&= \frac1{\|b\|_{L^\infty}^{\alpha}} \int_{\Omega\setminus K} \int_I 
	\left( \frac1\eta \int_0^\eta \chi_{A^\epsilon}(y){\rm dist}(y,\de A^\epsilon)^{\alpha} \dd\epsilon \right)
	f(v,y) \dd v \dd y .
\end{align*}
Since $\delta(y) \leq {\rm dist}(y,\de A^\epsilon) + \epsilon$, i.e.,  ${\rm dist}(y,\de A^\epsilon)\geq \delta(y)-\epsilon$,
we compute 
\begin{align*}
	\int_0^\eta \chi_{A^\epsilon}(y){\rm dist}(y,\de A^\epsilon)^{\alpha} \dd\epsilon
	&\leq \int_0^\eta \chi_{A^\epsilon}(y) (\delta(y)-\epsilon)^{\alpha} \dd\epsilon \\
	&= \int_0^{\min\{\eta,\delta(y)\}} (\delta(y)-\epsilon)^{\alpha} \dd\epsilon \\
	&\overset{(*)}\leq  \min\left\{2^{-\alpha} , \,\frac{2}{\alpha+1}\right\} \eta \delta(y)^{\alpha} .
\end{align*}
In $(*)$, we considered two cases: first, when $\eta<\delta(y)/2$, 
 the integral is bounded by $2^{-\alpha} \delta(y)^{\alpha}\eta$;
second, when $\eta\geq \delta(y)/2$, 
using the fact that $\alpha>-1$,
the integral is bounded by $\frac{2}{\alpha+1}\eta \delta(y)^{\alpha} $.

Therefore, there exists $\eta_0>0$ such that for all $\eta\in(0,\eta_0)$, we can estimate the averaged integral with
\[
\frac1\eta \int_0^\eta 
	\int_{A^\epsilon\setminus K}\int_I {\rm dist}(y,\de A^\epsilon)^{\alpha} f(v,y) \dd v\dd y
	\dd\epsilon
\leq
2^{-\alpha}
\int_{\Omega\setminus K} \int_I \delta(y)^{\alpha} f(v,y) \dd v \dd y .
\]
It follows that
\[
\liminf_{\epsilon\to0} \int_{A^\epsilon\setminus K}\int_I {\rm dist}(y,\de A^\epsilon)^{\alpha} f(v,y) \dd v\dd y 
\leq
2^{-\alpha}
\int_{\Omega\setminus K} \int_I \delta(y)^{\alpha} f(v,y) \dd v \dd y .
\]
Then, 
since we assumed $\int_{\Omega}\int_I \delta(y)^{\alpha} f(v,y) \dd v \dd y <\infty$ in~\eqref{eq:defLambda_p_bis}, 
we obtain the estimate~\eqref{eq63c7c693} and thus our claim~\eqref{eq63c7c610}.

We have thus completed the proof of~\eqref{eq63ea3e58}, and thus~\eqref{eq5ff85eabbis}.
%
Next we work on the second part of the theorem.
If~\eqref{eq63eb4e54} holds, i.e., $\hat\ell \le \ell(t,x)\le \ell$ for all $t\in I$ and $x\in\Omega$, then
\begin{equation}\label{eq63eb3f12}
	\Lambda''_p
	\leq \Lambda'_p
	\leq \left(\frac{\ell}{\hat\ell}\right)^{\frac{n}{p-n}} \Lambda''_p ,
\end{equation}
since $p\geq2n$.
We can repeat the steps in~\eqref{eq63c6ed39} and~\eqref{eq63eb3f4f}
without the integral in $s$, and then apply~\eqref{eq63ea3e58} and~\eqref{eq63eb3f12}
 to obtain
\begin{align*}
	&\int_{\Omega_{(t,s)}} \|D_xX(t,s,x)\|^p \dd x \\
	&\leq
	\left( 
		\int_{I}\int_{\Omega} 
		\frac{ \chi_{\Omega(s,v)}(y) }{ \ell_\Omega(s,X(s,v,y))^{q'} } \exp\left( \ell pq'\|D_xb(v,y)\| \right) 
	\dd y\dd v
	\right)^{1/q'}
	\times \\ &\phantom{AAAAAAAAAAAAAAAAAAAAA}\times
	\left(
	\int_{I}\int_\Omega
	 \|D_y X(s,v,y)\|^{nq}
	\dd y\dd v
	\right)^{1/q} 
	\leq
	\frac1{\hat\ell} \left(\frac{\ell}{\hat\ell}\right)^{\frac{n^2}{p(p-n)}} 
	 \Lambda''_p ,
\end{align*}
and this proves~\eqref{eq63ea3f57}.

If~\eqref{cptsuppb} holds, then the triple integral in~\eqref{eq63c6ec74} is finite also for $A=\Omega$ and $p>n$.
In this case, we can obtain directly~\eqref{eq63c7c112} for $A=\Omega$ without passing through the approximation $A^\epsilon$,
where we used the stronger hypothesis $p>2n$.
\end{proof}

The extension of Theorem~\ref{thm5e9a10ec_bis} to the case of Sobolev spatial regularity requires a
global approximation of $b$ by more regular vector fields that seems to be not trivial, also because of
the weight function $\ell(s,x)$ depending on the vector field itself. In addition, the non-doubling property
of the exponential function is  source of extra difficulties, when performing standard convolution arguments.
In the case of a weight independent of $t$, we addressed this problem in the note~\cite{ANGSC}, from which we extract the following result.

\begin{theorem}[{\cite[Theorem~3 and Remark~25]{ANGSC}}]\label{thm:MS} 
Let $\Omega$ be a bounded open set and let $w:\Omega\to (0,+\infty)$, with $w+w^{-1}\in L^\infty_\loc(\Omega)$.
\begin{itemize}
\item[(i)] If $b\in \Lbs^1(I;W^{1,1}_\loc(\Omega;\R^n))$ satisfies 
\begin{equation}\label{expintapprox}
\int_I\int_\Omega w(x)\exp(c\|D_x b(s,x)\|)\dd x\dd s<+\infty,
\end{equation}
for some $c>0$, then there exist $b_h\in C^\infty(I\times\Omega;\R^n)$ satisfying, whenever $\Omega'\Subset\Omega$,
\begin{equation}\label{bhtobunif}
\text{$b_h$ converge to $\tilde{b}$ in $L^1(I;C(\Omega';\R^n))$ as $h\to\infty$}\,,
\end{equation}
where $\tilde b$ is the space continuous representative of $b$,
\begin{equation}\label{DxbhtoDxbL1}
D_xb_h\to D_xb\text{ in }L^1(I;L^1(\Omega';\R^n)), \text{ as $h\to\infty$}\,,
\end{equation}
and 
\begin{equation}\label{convexpenergy}
\lim_{h\to\infty}\int_I\int_\Omega w(x)\exp(c\|D_x b_h(s,x)\|)\dd x\dd s=
\int_I\int_\Omega w(x)\exp(c\|D_x b(s,x)\|)\dd x\dd s.
\end{equation}
\item[(ii)]
If $ b\in L^1(I;W^{1,1}_{\rm loc}(\R^n;\R^n))$  
and there exists a bounded open set $\Omega$ such that
\begin{equation}\label{bcptsupp}
{\rm spt}(b(t,\cdot))\subset \Omega\text{ for each }t\in I\,,
\end{equation}
and $b$ satisfies \eqref{expintapprox} on $\Omega$,
then there exist $b_h\in C^\infty(I\times\R^n;\R^n)$ still satisfying \eqref{bhtobunif}, \eqref{DxbhtoDxbL1},  \eqref{convexpenergy} and also
\begin{equation}\label{bhcptsupp}
{\rm spt}(b_h(t,\cdot))\subset \Omega\text{ for each }t\in I \text{ and }h\in\N\,.
\end{equation}
\end{itemize}
\end{theorem}

\begin{proof}[Proof of Theorem~\ref{thm5e9a10ec}] 
Let
\[
w(x):=\max\left\{
	\ell^{n/(n-p)},\frac{({\rm dist}(x,\partial\Omega))^{n/(n-p)}}{(\sup |b|)^{n/(n-p)}}
	\right\},
\]
so that, with $c=\,\ell p^2/(p-n)$, one has
$$\int_I\int_\Omega w(x)\exp(c\|D b(s,x)\|)\dd x\dd s<+\infty$$
and we may apply Theorem~\ref{thm:MS} (i) to $b$. Since $b$ is bounded, we can also
assume, by a truncation argument, that $\sup |b_h|\leq\sup |b|$. 
In addition, the quantity $\Lambda_p'$ in~\eqref{eq63eb586e}
is finite and satisfies
$\Lambda_p'\leq \Lambda_p$, using the inequality~\eqref{eq6283574e}.

It follows from these considerations that 
$$
\limsup_{h\to\infty}\Lambda_{p,h}'\leq \limsup_{h\to\infty}\Lambda_{p,h}\leq  \Lambda_p<+\infty\,,
$$
where
\begin{equation}\label{Lambdahprime}
\Lambda_{p,h}':=\int_I\int_\Omega (\ell_h(s,x))^{n/(n-p)}\exp\biggl(\frac{\ell p^2}{p-n}\|D_xb_h(s,x)\|\biggr)\dd x \dd s\,,
\end{equation}
and similarly $\Lambda_{p,h}$ is defined with $b_h$.

Now, in order to apply \eqref{eq63ea3e58} from Theorem~\ref{thm5e9a10ec_bis} to $b_h$ and the pass to the limit $h\to\infty$, it suffices by Fatou's lemma to prove that 
\begin{equation}\label{eq63eb631f}
\int_{\Omega_{(t,s)}} \|D_xX(t,s,x)\|^p \dd x\leq\liminf_{h\to\infty}\int_{\Omega^h_{(t,s)}} \|D_xX^h(t,s,x)\|^p \dd x\,,
\end{equation}
where $\Omega^h_{(t,s)}$ and $X^h$ are relative to the vector fields $b_h$. Now, \eqref{eq:domains}, derived from
Lemma~\ref{lem5e996a0e}, yields that
for any open domain $A\Subset\Omega_{(t,s)}$ one has $A\Subset\Omega^h_{(t,s)}$ for $h$ large enough, and that
$X^h(t,s,\cdot)$ converge to $X(t,s,\cdot)$ uniformly on $A$. Hence the lower semicontinuity of
$w\mapsto\int_A|D_xw|^p\dd x$ yields
$$
\int_A \|D_xX(t,s,x)\|^p \dd x\leq\liminf_{h\to\infty}\int_A \|D_xX^h(t,s,x)\|^p \dd x\leq\liminf_{h\to\infty}
\int_{\Omega^h_{(t,s)}} \|D_xX^h(t,s,x)\|^p \dd x\,.
$$
Letting $A\uparrow\Omega_{(t,s)}$ we obtain the claimed semicontinuity property~\eqref{eq63eb631f}.
\end{proof}

\begin{proof}[Proof of Theorem \ref{thmbcpcsupp}]
First assume that $b$ is understood as vector field $b:\,I\times\Omega\to\R^n$. Let $w\equiv 1$
so that, with $c=\,\ell\,p^2/(p-n)$, one has
$$\int_I\int_\Omega \exp(c\|D b(s,x)\|)\dd x\dd s<+\infty\,.$$
From Theorem~\ref{thm:MS} (ii),  let $b_h:\,I\times\Omega\to\R^n$ be the regular sequence of vector fields 
satisfying \eqref{bhtobunif}, \eqref{DxbhtoDxbL1}, \eqref{convexpenergy} and \eqref{bhcptsupp}.

It follows from \eqref{bcptsupp} that $\ell_h(t,x)=\ell$ and thus
\[
\limsup_{h\to\infty}\Lambda_{p,h}'
\leq   
\frac1{\ell}
\int_I\int_\Omega \exp(c\|D b(s,x)\|)\dd x\dd s<+\infty\,,
\]
where $\Lambda_{p,h}'$ is the quantity in \eqref{Lambdahprime}.

Now, in order to apply \eqref{eq63ea3f57} with $p>n$ from Theorem~\ref{thm5e9a10ec_bis} to $b_h$
and then pass to the limit $h\to\infty$, we only need to use~\eqref{eq63eb631f} again.


Assume now that $b$ is understood as vector field $b:\,I\times\R^n\to\R^n$. It is clear that the flow map $X:\,I\times I\times\R^n\to\R^n$ is identically  equal to the identity on  $I\times I\times (\R^n\setminus\overline{\Omega'})$. Since \eqref{eq5e9998eb} holds for any bounded open set $\Omega\Supset\Omega'$, then
 $X(t,s,\cdot)\in W^{1,p}_{\rm loc}(\R^n)$.
\end{proof}

\begin{corollary}\label{cor5ec7b067}
Under the assumptions of Theorem~\ref{thm5e9a10ec},  we have, for every $t\in I$,
\begin{equation}\label{eq:areast}
\begin{aligned}
	&\text{for almost every $s\in I$}\\
	&X(t,s,\cdot)_\# \Lbm^n\res\Omega_{(t,s)} = J_X(s,t,\cdot) \Lbm^n\res\Omega_{(s,t)}=
	\,\frac{1}{J_{X(t,s,X(s,t,\cdot))}}\Lbm^n\res\Omega_{(s,t)}\,,
\end{aligned}
\end{equation}
where $J_X(t,s,\cdot) = \det(D_x X(t,s,\cdot)) \in \Lbs^{p/n}(\Omega_{(t,s)})$ is non-zero almost everywhere. Under the assumptions of Theorem~\ref{thmbcpcsupp}, then \eqref{eq:areast} holds for every $s,\,t\in I$ replacing both $\Omega_{(t,s)}$ and $\Omega_{(s,t)}$ by $\R^n$.
\end{corollary}

In fact, $J_X$ is strictly positive, as we will show in the next corollary.

\begin{proof}
	Since $X(t,s,\cdot)^{-1} = X(s,t,\cdot)$ is also in $W^{1,p}_\loc(\Omega_{(s,t)};\R^n)$ with $p>2n$, then both maps $X(t,s,\cdot)$ and $X(t,s,\cdot)^{-1}$ are differentiable almost everywhere by Lemma~\ref{lem5e9a1230}.
	Therefore, we have $J_X(t,s,x)\neq0$ by Lemma~\ref{lem5e9a12a0}.
	By Lemma~\ref{lem5e9a1230}, $X(t,s,\cdot)$ satisfies Lusin's (N) condition, and thus, by Lemma~\ref{lem04031003}, the area formula holds, that is, \eqref{eq:areast} holds. Moreover, by the Laplace expansion of the determinant, one can easily prove, by induction on the order matrix and H\"older inequality, that  $J_X(t,s,\cdot) = \det(D_x X(t,s,\cdot)) \in \Lbs^{p/n}(\Omega_{(t,s)})$.  The same argument applies under the assumptions of Theorem \ref{thmbcpcsupp}.
\end{proof}

\begin{corollary}\label{cor5edfe500}
	Under the assumptions of Theorem~\ref{thm5e9a10ec}, for every $(s,x)\in I\times\Omega$, if we set
	\begin{align*}
	y(t)&:=D_xX(t,s,x) , &
	B(t)&:=(D_xb)(t,X(t,s,x)) , \\
	J(t)&:=\det(D_xX(t,s,x)) , &
	\beta(t)&:=\div_x b(t,X(t,s,x)) = \trace(B(t)),
	\end{align*}
	then $y$ and $J$ are absolutely continuous solutions to the initial value problems
	\begin{align}
		&\label{align5e9a1b36}
		\begin{cases}
			\dot{y}(t) = B(t) y(t) , \\
			y(s)= \mathrm{Id} .
		\end{cases} \\
		&\label{align5e9a1b40}
		\begin{cases}
			\dot{J}(t)=\beta(t)J(t) , \\
			J(s)=1 .
		\end{cases}
	\end{align}
	Moreover, for almost every $(t,s,x)\in\Dom_b$ we have
	\begin{align}
	\label{eq5e9a1b60}
		D_xX(t,s,x) &= \exp\left( \int_s^t D_x b(v,X(v,s,x)) \dd v \right) , \\
	\label{eq5e9a1e0a}
		J_X(t,s,x) &= \exp\left( \int_s^t \div_xb(v,X(v,s,x)) \dd v \right) .
	\end{align}
	In particular, $J_X>0$ almost everywhere.
\end{corollary}
\begin{proof}
	First, we claim that, given $s$,  for almost every $x$ the matrix $B$ belongs to $\Lbs^1(I_{(s,x)};\R^{n^2})$, where $I_{(s,x)}$ is defined in Section~\ref{sec5ed296d6}.
	Indeed, the change of variables $z=X(t,s,x)$ and the identity
	$J_X(s,t,X(t,s,x)) = 1/J_X(t,s,x)$ give
	\[
	\int_I\int_{\Omega_{(s,t)}} \| (D_xb)(t,X(t,s,x)) \| \dd x\dd t
	= \int_I\int_{\Omega_{(t,s)}} \| (D_xb)(t,z) \| J_X(s,t,z) \dd z\dd t .
	\]
	The latter integral is finite because $D_xb\in \Lbs^q$ for all $q$ and $J_X\in \Lbs^{p/n}$.
	Therefore, the claim is true.
	
	Second, using the same approximation of $b$ as in the proof of Theorem~\ref{thm5e9a10ec}, 
	we know that $D_x X^h\to D_x X$ weakly in $\Lbs^p(\Dom_b)$, since
	$D_x X^h$ are uniformly bounded bounded in $\Lbs^p(\Dom_b)$.
	
	Third, 
	we see that the distributional derivative $\de_tD_x X$ has the following form: for every $\phi\in C^\infty_c(\Dom_b;\R^n)$,
	\begin{align*}
	\de_tD_x X [\phi]
	&= - \int_{\Dom_b} D_x X \de_t\phi \dd t\dd s \dd x \\
	&= - \lim_{\epsilon\to0} \int_{\Dom_b} D_x X_\epsilon \de_t\phi \dd t\dd s \dd x \\
	&=  \lim_{\epsilon\to0} \int_{\Dom_b} D_xb_\epsilon(t,X_\epsilon(t,s,x)) D_x X_\epsilon(t,s,x) \phi \dd t\dd s \dd x .
	\end{align*}
	Now, $X_\epsilon\to X$ uniformly on compact sets and $D_xb_\epsilon\to D_xb$ in $\Lbs^q_\loc(\Dom_b)$ for all $q$.
	So, using Hölder inequality and the Lebesgue dominated convergence theorem, we obtain that the limit above is 
	\[
	\de_tD_x X [\phi]
	= \int_{\R^{2+n}} D_xb(t,X(t,s,x)) D_x X(t,s,x) \phi \dd t\dd s \dd x .
	\]
	In other words, $\de_tD_x X = D_xb(t,X(t,s,x)) D_x X(t,s,x)$.
	This shows that $y$ is solution to the Cauchy system~\eqref{align5e9a1b36}.
	Since $B$ is integrable, then we get~\eqref{eq5e9a1b60} by integrating this Cauchy system.
	
	Finally, the validity of~\eqref{align5e9a1b40} follows in a standard way from \eqref{align5e9a1b36} and
	\eqref{align5e9a1b40} implies~\eqref{eq5e9a1e0a}.
\end{proof}

\begin{corollary}\label{cor63edf816}
	Under the assumptions of Theorem~\ref{thm5e9a10ec},
	for every $1<r<p-n$ and for every $s,t\in\overline{I}$, 
	we have that
	\begin{equation}\label{eq63ebf552}
		X(t,s,\cdot)\in W^{1,r}(\Omega_{(t,s)};\R^n) .
	\end{equation}
\end{corollary}
\begin{proof}
	The proof is an improvement of~\eqref{eq5ff85eab} through an application of
	Corollary~\ref{cor5edfe500}, Lemma~\ref{lem5ec7d831} and Proposition~\ref{prop5ec7f3d8}.
	Indeed, for every $1<r<p-n$ there exists $0<q<p^2-\frac{p^2}{p-n}$ such that
	\[
	r = \frac{p^2(p-n)}{q(p-n)+p^2},
	\quad\text{ that is, }\quad
	\frac{r}{p-r} = \left(\frac{q}{p} + \frac{n}{p-n}\right)^{-1}.
	\]
	So, given $s,t\in I$,~\eqref{eq5ff85eab} implies that there exists $u\in I$ such that 
	\[
	X(s,t,\cdot) = X(s,u,X(u,t,\cdot)) 
	\]
	and both maps $X(s,u,\cdot)$ and $X(u,t,\cdot)$ belong to $W^{1,p}$ on their domains,
	with non-zero Jacobian by Corollary~\ref{cor5edfe500}.
	We then apply Lemma~\ref{lem5ec7d831} and Proposition~\ref{prop5ec7f3d8} to prove that their composition is of class $W^{1,r}$ on its domain.
\end{proof}

\begin{corollary}\label{cor5ec2a4b4}
	Under the assumptions of Theorem~\ref{thm5e9a10ec}, we have
	\[
	X\in W^{1,p} (\Dom_b;\R^n) .
	\]
\end{corollary}
\begin{proof}
	We know that $X$ is continuous and, from~\eqref{eq63eb5e33}, we have that $D_x X\in \Lbs^p(\Dom_b;\R^{n^2})$.
	Moreover, by the identity $X(t,s,x)=x+\int_s^t b(v,X(v,s,x))\dd v$, we have
	\[
	\de_t X(t,s,x) = b(t,X(t,s,x)), 
	\]
	and thus $\de_t X$ is continuous. Finally,  by differentiating with respect to $s$ the semigroup identity~\eqref{eq5ec5219e} in the form
	$X(t,v,x)=X(t,s,X(s,v,x))$, we get
	\[
	\de_sX(t,s,X(s,v,x))+D_xX(t,s,X(s,v,x))b(s,X(s,v,x))=0,
	\]
	for all $t,s,v\in\R$ and $x\in\R^n$ for which the expression make sense.
	Therefore, $\de_sX(t,s,y) = - D_xX(t,s,y)b(s,y)$ and so $\de_sX\in\Lbs^p(\Dom_b;\R^n)$.
\end{proof}

\begin{remark}
Corollary~\ref{cor5ec2a4b4} improves \cite[Theorem 4]{MR3912727} and \cite[Corollary 1.8]{MR4263701}, by dropping assumption~\eqref{divbbd},
that is $\div_x b\in \Lbs^1_\loc(\R ,\Lbs^\infty(\R^n ))$.
Notice that, if $n=1$, assumption~\eqref{divbbd} reduces to the classical Lipschitz condition of $b$ with respect to $x$, uniformly in $t$.  
Corollary~\ref{cor5ec2a4b4}  also applies  to a non Lipschitz one-dimensional vector field $b$ (see Example~\ref{sec5ec2a1ee} below).
\end{remark}

\begin{remark}\label{rem5ed2a008}
Corollary~\ref{cor5ec2a4b4} looks almost sharp. 
Indeed, it was proved in \cite{MR3437603} that no Sobolev regularity can be expected for the flow, when assuming only that $b\in\Lbs^1(I;W^{1,p}(\R^n;\R^n))$ for all finite $p\in [1,+\infty)$, even when $b$ is compactly supported and divergence-free. 
See also Remark~\ref{rem5ede0b6b}.
\end{remark}

\section{Applications to PDEs}\label{sec5edff3f5}
In this section we apply the Sobolev regularity of flows for getting the representation of  weak solutions of the Cauchy problems both for the transport and continuity equations.

\begin{proof}[Proof of Theorem~\ref{solTE}]
By \cite[Theorem 1]{MR3458193}, we can infer the uniqueness of weak solutions $u\in L^\infty((0,T);L^\infty(\R^n))$  for~\eqref{CPTE}, provided that the spatial derivative of $b$ satisfies a sub-exponential summability and $\bar u\in L^\infty(\R^n)$.
Therefore we have only to show that the function $v$ in~\eqref{represcharTE} is a weak solution of~\eqref{CPTE}, that is, 
\begin{equation}\label{vws}
\int_0^T\int_{\R^n}v\,\left(\partial_t\varphi+\div(b\,\varphi)\,\right)dtdx=\,-\int_{\R^n}\bar u\,\varphi(0,\cdot)\dd x
\end{equation}
for each $\varphi\in C^\infty_c([0,T)\times\R^n)$.
We divide the proof in two steps.

{\bf 1st step.} 
Let us assume that $\bar u\in  C^\infty(\R^n)\cap L^\infty(\R^n)$. 
Let $b_\epsilon: I\times\R^n\to \R^n$ be the family of vector fields defined as
\[
b_\epsilon(t,x):=\,(\tilde b_\epsilon(t,\cdot)*\rho_\epsilon)(x)\text{ if }(t,x)\in I\times\R^n\,,
\] 
where
\[
 \tilde b_{\epsilon}(t,x):=\,(\tilde b_{\epsilon,1}(t,x),\dots,\tilde b_{\epsilon,n}(t,x))\,,
 \]
\[
\tilde b_{\epsilon,i}(t,x):=\,\max\{\{\min\{b_i(t,x),1/\epsilon\},-1/\epsilon\}\text{ if }(t,x)\in I\times\R^n,\,i=1,\dots,n\,,
\]
and $(\rho_\epsilon)_\epsilon$ denotes a family of mollifiers depending on the space variable $x$. Then,
standard properties of convolutions yield
\begin{equation}\label{bepsbd}
|b_\epsilon(t,x)|\le\,\frac{\sqrt n}{\epsilon}\text{ for each }(t,x)\in I\times\R^n\,,
\end{equation}
\begin{equation}\label{bepsreg}
b_\epsilon(t,\cdot)\in C^\infty(\R^n;\R^n)\text{ and }D_xb_\epsilon(t,x)=
\,\int_{\R^n}D_x\rho_\epsilon(x-y)\,\tilde b_\epsilon(t,y)\,dy\quad\forall (t,x)\in I\times\R^n\,,
\end{equation}
\begin{equation}\label{bepswp}
\text{ $b_\epsilon(t,\cdot):\R^n\to\R^n$ is Lipschitz continuous, uniformly with respect to $t\in I$},
\end{equation}
\begin{equation}\label{bepstobunif}
b_\epsilon\to b\quad\text{in $L^1(I;L^1(\Omega;\R^n))$ as $\epsilon\to 0$, whenever $\Omega\Subset\R^n$}\,.
\end{equation}

Then, from \eqref{bepswp} we get that $b_\epsilon $ is well-posed and the flow maps  $X_\epsilon$ associated  to $b_\epsilon$ is globally defined, that is, $X_\epsilon:\, \overline I\times\R^n\to\R^n$ and is Lipschitz regular.
Define
\[
v_\epsilon(t,x) = \bar u(X_\epsilon(0,t,x)) .
\]

By \eqref{bepstobunif},  Lemma~\ref{lem5e996a0e} and the subsequent remark, $X_\epsilon\to X$ uniformly on compact subsets 
of $\overline I\times\R^n$.
Since $\bar u$ is assumed to be continuous, we obtain that $v_\epsilon\to v$ uniformly on compact sets of $\overline I\times\R^n$. Moreover, since  $\bar u\in L^\infty(\R^n)$,  by the dominated convergence theorem, we can also assume that $v_\epsilon\to v$ in $L^2$ on each compact set of $\overline I\times\R^n$.

By the classical Cauchy-Lipschitz theory (see, for instance, \cite[Section 2]{MR3283066}), it is well-known that $v_\epsilon$ 
is a classical solution to~\eqref{CPTE} with $b_\epsilon$ in place of $b$ and, in particular, a weak solution, i.e., 
\[
\int_0^{T}\int_{\R^n}
	v_\epsilon \left(\partial_t\varphi+\div(b_\epsilon\varphi)\right)
	\dd x\dd t
	= - \int_{\R^n} \bar u(x) \varphi(0,x) \dd x \,,
\]
for each $\varphi\in C^\infty_c([0,T)\times\R^n)$. Since $\varphi$ has compact support in $[0,T)\times\R^n$, it is easy to check that
$\div(b_\epsilon\varphi)\to\div(b\varphi)$ in $L^2$. Hence
\[
\lim_{\epsilon\to0^+} 
	\int_0^T\int_{\R^n}
	v_\epsilon \left(\partial_t\varphi+\div(b_\epsilon\varphi)\right)
	\dd x\dd t
= \int_0^{T}\int_{\R^n}
	v \left(\partial_t\varphi+\div(b\varphi)\right)
	\dd x\dd t ,
\]
that is, $v$ is a weak solution to~\eqref{CPTE}. 

{\bf 2nd step.} 
Let $\bar u\in \Lbs^\infty(\R^n)$ and, by mollification in $\R^n$, let $\bar u_j\in\Lbs^\infty(\R^n)\cap C^\infty(\R^n)$, with $j\in\N$, be a sequence of functions that converges to $\bar u$ almost everywhere and such that $\|\bar u_j\|_{\Lbs^\infty}\le\|\bar u\|_{\Lbs^\infty}$ for every $j$.

Define
\[
v_j(t,x) = \bar u_j(X(0,t,x)) .
\]
By Theorem~\ref{thm5eb900db}, for every $t$ the homeomorphism $X(0,t,\cdot)$ satisfies the Lusin (N) condition, 
and therefore $v_j(t,\cdot)\to v(t,\cdot)$ almost everywhere in $\R^n$.
Since this is true for every $t$, we get that $v_j\to v$ almost everywhere in $[0,T]\times\R^n$.
Moreover, it is clear that $\|v_j\|_{\Lbs^\infty}\le\|\bar u\|_{\Lbs^\infty}$ for all $j$.

By the previous step, we have, for all $j$,
\[
\int_0^{T}\int_{\R^n}
	v_j \left(\partial_t\varphi+\div(b\varphi)\right)
	\dd x\dd t
	= - \int_{\R^n} \bar u_j(x) \varphi(0,x) \dd x \,,
\]
for each $\varphi\in C^\infty_c([0,T)\times\R^n)$. We can now apply the Dominated Convergence Theorem and pass to the limit $j\to\infty$ to obtain that $v$ is a weak solution to~\eqref{CPTE}.
\end{proof}

\begin{corollary}[Sobolev regularity of the solutions of the transport equation]\label{cor5ec7d875}
Let $b$ $\in C^0(I\times\R^{n};\R^n)$ as in Theorem~\ref{thmbcpcsupp} 
and let $X$ be the flow of $b$. Let $p>2n$ and $1\le\tilde q\le q<\infty$ be such that
\begin{equation}\label{eq5fce4500}
\frac{\tilde q}{q-\tilde q} = \left( \frac{q}{p} + \frac{n}{p-n} \right)^{-1}  
,\quad\text{i.e.,}\quad
\tilde q = \frac{ pq(p-n) }{ q(p-n)+p^2 }\,.
\end{equation}
If $\bar u\in \Lbs^\infty(\R^n)\cap W^{1,q}_\loc(\R^n)$,
then the function $v$ in~\eqref{represcharTE} satisfies 
\[
v\in L^\infty([0,T];W^{1,\tilde q}_\loc(\R^n))\, .
\]
\end{corollary}
\begin{proof}
By Theorem~\ref{thmbcpcsupp}, the function $\Psi:=X(0,t,\cdot)$ and its inverse are in $W^{1,p}_\loc$ with $p>2n$.
By Lemma~\ref{lem5ec7d831} and Corollary~\ref{cor5edfe500}, $\Psi$ is of finite distortion and $K^\Psi_q\in\Lbs^r_\loc$ for $r=\left( \frac{q}{p} + \frac{n}{p-n} \right)^{-1}$.
By Proposition~\ref{prop5ec7f3d8}, the composition operator $T_\Psi$ is continuous 
from $W_\loc^{1,q}(\Omega_2)$ to $W_\loc^{1,\tilde q}(\Omega_1)$.
Since $v(t,\cdot)=T_\Psi(\bar u)$, the proof is concluded.
\end{proof}

\begin{remark}\label{rem5edfebb7}
Notice that the propagation of regularity, in the spirit of Corollary~\ref{cor5ec7d875}, may fail below the exponential summability of $D_xb$, even though $\bar u\in C^\infty_c(\R^n)$. 
Indeed,  in \cite[Theorem~2.1]{MR4263701}, the authors constructed 
a divergence-free vector field $b:\R \times\R^n\to\R^n$ ($n\ge\,2$)
satisfying the subexponential summability condition~\eqref{eq5eb95263},
and a weak bounded, compactly supported solution $u(t,x)$ of~\eqref{CPTE} such that 
$\bar u:=\,u(0,\cdot)\in C^\infty_c(\R^n)$
but $u(t,\cdot)\notin \dot{W}^{s,p}(\R^n)$ for all $t>0$, $s>0$ and $p\ge1$,
where $\dot{W}^{s,p}(\R^n)$ denotes the so-called {\it homogeneous Sobolev space}.
The example is based on the work \cite{MR3933614}.  
Let us recall that when $s=1$  and $1<\,p<\,\infty$, then $\dot{W}^{1,p}(\R^n)\cap L^p(\R^n)$  coincides with the classical Sobolev space ${W}^{1,p}(\R^n)$ (see \cite[Section 2]{MR3933614}). 

Let us point out that, although $b$ satisfies the hypothesis of Theorem~\ref{solTE},
$u(t,\cdot)\notin {W}^{1,p}(\R^n)$ for each $t\in (0,\infty)$ and $p\in (1,\infty)$.
This implies that the flow of $b$ has not Sobolev regularity $W^{1,p}$ for some $p>n$, otherwise the same proof of Corollary~\ref{cor5ec7d875} could be repeated.
Our Example~\ref{sec5eb962e5} shows the same phenomenon.
Notice that we don't know whether $u(t,\cdot)\notin {W}^{1,1}(\R^n)$.
\end{remark}

\begin{remark}
	Corollary~\ref{cor5ec7d875} shows that the Sobolev regularity of the flow $X$ implies Sobolev regularity for solutions of the transport equation in the form~\eqref{represcharTE}.
	We notice that the converse implication is almost true.
	Indeed, suppose that for every $1\le\tilde q\le q<\infty$ satisfying~\eqref{eq5fce4500},
	and for every $\bar u\in \Lbs^\infty(\R^n)\cap W^{1,q}_\loc(\R^n)$,
	the function $v$ in~\eqref{represcharTE} satisfies 
	$v\in L^\infty((-T,T);W^{1,\tilde q}_\loc(\R^n))$.
	Then, if we take $\bar u = \phi x_j$, where $\phi\in C^\infty_c(\R^n)$ and $x_j$ a the $j$-th coordinate function,
	we see that $x_j(X(0,t,\cdot))\in W^{1,\tilde q}_\loc(\R^n)$ for almost every $t$ and for all $\tilde q<p$.
\end{remark}

Before proving Theorem~\ref{solCE}, let us introduce some notation and recall some existence and uniqueness results as concern as the Cauchy problem \eqref{CPCE} for the continuity equation.  
It is well-known that the non trivial issue in the well-posedness turns out  to be  the uniqueness of weak  solutions.  Nonnegative measure-valued solutions of the continuity equation are uniquely determined by their initial condition if the characteristic ODE associated to the vector field $b$ has a unique solution (see \cite[Theorem 9]{MR3283066}). A partial extension  of this result to signed measures was given in \cite{MR2439520}, under a quantitative two-sided Osgood condition on  $b$. For the proof of Theorem~\ref{solCE} we will use the following more general result.
\begin{theorem}[{\cite[Theorem 3]{MR3906270}}] \label{wpCE}Let $b:\,(0,T)\times\R^n\to\R^n$ be such that the following two conditions hold: There exists a continuous and nondecreasing function $G:\,[0,\infty)\to  [0,\infty)$ satisfying $\int_r^\infty\frac{ds}{G(s)}ds=\infty$ for some $r>\,0$ such that
\begin{equation}\label{B1}
\sup_{x\in \R^n}\frac{|b(t,x)|}{G(|x|)}\in L^\infty(0,T)\\,
\end{equation}
and there exists a continuous and nondecreasing function $\omega :\,[0,\infty) \to [0,\infty)$  satisfying $\int_0^r\frac{ds}{\omega(s)}ds=\infty$ for some $r>\,0$ such that
\begin{equation}\label{B2}
\sup_{x,\,y\in B(0,R),\,x\neq y}\frac{|b(t,x)-b(t,y)|}{\omega(|x-y|)}\in L^\infty(0,T)\,,
\end{equation}
for any radius $R>\,0$. Then for any $\rho_0\in\mathcal M(\R^n)$ there exists a unique solution $\rho$ in $ L^1((0,T);\mathcal M(\R^n))$ for the Cauchy problem \eqref{CPCE}, that is,
\begin{equation}\label{defweaksolCPCEmeas}
\int_0^{T}\int_{\R^n}\left(\partial_t\varphi(t,x)+\langle b(t,x),D_x\varphi(t,x)\rangle\right)\,d\rho_t(x)dt=\,-\int_{\R^n}\varphi(0,x)\,d\rho_0(x)
\end{equation}
for each $\varphi\in C^\infty_c([0,T)\times\R^n)$, with $\rho_t:=\,\rho(t,\cdot)$. Moreover  this  solution representable as 
$\rho_t=\,X(t,0,\cdot)_\#\rho_0$, where $X(t,s,\cdot)$ denotes the flow of the vector field $b$.
\end{theorem}
\begin{remark}\label{rmkclop} As pointed out in \cite[p. 49]{MR3906270}, by assuming $\rho\in L^\infty((0,T);\mathcal M(\R^n))$, the uniqueness still holds if one replaces $L^\infty(0,T)$  by $L^1(0,T)$ in both \eqref{B1} and \eqref{B2}.
\end{remark}
We will also use  the following abstract area-type formula, which can be proved in a standard way.
\begin{lemma}\label{represpfmeas} Let $\Psi:\,\R^n\to\R^n$, $\mu:\mathcal M_n\to [0,\infty]$ and $\bar\rho:\,\R^n\to [-\infty,\infty]$  be a homeomorphism, a Radon measure on $\R^n$ and a function in $ L^1(\R^n,\mu)$, respectively. Suppose that $\Psi_\#\mathcal\mu$  is absolutely continuous with respect to $\mu$ and denote
\[
w:=\frac{d\Psi_\#\mu}{d\mu}\in L_{\rm loc}^1(\R^n,\mu)\,.
\]

Then
\[
\Psi_\#(\bar\rho\mu)=\,\bar\rho(\Psi^{-1})\,w\,\mu\,.
\]
In particular it also follows that
\[
\bar\rho(\Psi^{-1})\,w\in L^1(\R^n,\mu)\,.
\]
\end{lemma}

\begin{proof}[Proof of Theorem~\ref{solCE}] (i)  In view of the use of Remark~\ref{rmkclop}, 
let us begin to show that \eqref{B1} and \eqref{B2} hold
with $L^1((0,T))$ in place of $L^\infty((0,T))$. It is immediate that, by choosing
\[
G(s):=\,s\,\log^+s\text{ if }s\ge\,0\,,
\]
\eqref{B1}  in the weaker form follows from \eqref{globexsistb}. As for \eqref{B2}, observe that, from \eqref{eq5eb95263vecchia}, it follows that
\begin{equation}\label{estimB2}
\int_0^T\int_{B(0,R)}\exp\left( \frac{ \|D_xb\| }{ \log^+\|D_xb\| } \right)\,dxdt<+\infty\text{ for each }R>\,0\,.
\end{equation}
Let $\theta:[0,\infty)\to [0,\infty)$ be the function defined as
\[
\theta(s):=
\begin{cases}\exp\displaystyle{\left(\frac{\bar s}{\log \bar s}\right)}&\text{ if }0\le\,s\le\,\bar s\,;\\
\exp\displaystyle{\left( \frac{s }{ \log s } \right)}&\text{ if }s\ge\,\bar s\,.
\end{cases}
\]
Then, by Proposition \ref{propFEL},  $\theta$ satisfies the assumption of Theorem \ref{thm5e999915}, for $\bar s$ large enough. In particular, we can deduce, as in the proof of Theorem \ref{thm5e999915}, the estimate
\[
\sup_{x,\,y\in B(0,R),\,x\neq y}\frac{|b(t,x)-b(t,y)|}{\omega(|x-y|)}\le\,\varphi(t)
\]
 where $\omega$ and $\varphi$ are the modulus of continuity defined in  \eqref{modcontomega} and the function in \eqref{functmodcontomega}, respectively. From \eqref{estimB2},  it follows that $\varphi\in L^1(0,T)$. Thus \eqref{B2} in the weaker form follows, too. Hence, the uniqueness of weak solutions for \eqref{CPCE} in  $L^\infty((0,T);\mathcal M(\R^n))$ is granted by Theorem~\ref{wpCE} and Remark~\ref{rmkclop} .

The existence  of weak solutions for \eqref{CPCE} can be proved as in \cite{MR2439520,MR3906270}. More precisely, if $X:\,[0,T]\times[0,T]\times\R^n\to\R^n$ denotes the flow associated to $b$  and $\rho$ is defined as in  \eqref{solCPCEmeas}, then   one can verbatim repeat the same proof, by showing that $\rho$   is a weak solution of \eqref{CPCE} in  $L^\infty((0,T);\mathcal M(\R^n))$, that is, \eqref{defweaksolCPCEmeas} holds.

If now $\bar \rho\in L^1(\R^n)$, we can apply the previous result with initial value measure $\bar\rho\,\mathcal L^n$ and we get that the unique weak solution for \eqref{CPCE} is given by
\[
\rho_t=\,X(t,0,\cdot)_\#(\bar\rho\,\mathcal L^n)\text{ for each }t\in [0,T]\,.
\]
From Theorem~\ref{thm5eb900db}, $\rho_t=\,X(t,0,\cdot)_\#\mathcal L^n$ is absolutely continuous with respect to  $\mathcal L^n$. Therefore, by Lemma \ref{represpfmeas} with $\Psi=\,X(t,0,\cdot)$ and $\mu=\,\mathcal L^n$, \eqref{eq6033caa8} follows.

\noindent (ii) Observe that  in this case the assumptions of Theorem~\ref{thm5eb900db} are still satisfied. Thus, by applying the previous claim (i), it follows that \eqref{eq6033caa8} holds. We have only to prove that
\begin{equation}\label{w=jacob}
J_{X,t}=\,J_{X(0,t,\cdot)}\text{ a.e.~in }\R^n\text{, for each }t\in [0,T]\,
\end{equation}
in order to conclude the proof.
From Theorem \ref{thmbcpcsupp}, the flow $X(t,s,\cdot)\in W^{1,p}_{\rm loc}(\R^n;\R^n)$ with $p>\,n$, for each $t,\,s\in [0,T]$. In particular, from Corollary~\ref{cor5ec7b067},  we get that
\[
\,X(t,0,\cdot)_\#(\mathcal L^n)=\,J_{X(0,t,\cdot)}\mathcal L^n\,.
\]
Thus \eqref{w=jacob} follows.
\end{proof}

\section{Examples}\label{sec5edff41e}
Most of the examples we provide are in low dimension.
However, notice that we can always extend the examples to higher dimensions as follows.
If $b(t,x)$ is a vector field on $\R^n$ and $m>0$, then define the vector field $h(t,(x,y)):=(b(t,x),0)$ on $\R^n\times\R^m$.
If $X(t,s,x)$ is the flow of $b$, then 
$Z(t,s,(x,y)) := (X(t,s,x),y)$ is the flow of $h$.
Notice that $Z(t,s)\in W^{1,p}_\loc(\R^n\times\R^m)$ if and only if $X(t,s)\in W^{1,p}_\loc(\R^n)$.

\subsection{Well-posedness does not imply absolute continuity of the flow}\label{sec5ec2a20e}
Here we show that
{\it there is a vector field $b\in C^0(\R\times\R)$ that is well-posed, but whose flow on $\R$ is not absolutely continuous.}
This shows that Theorems~\ref{thm5eb900db} and~\ref{thm5ea9862b} are not a direct consequence only of the well-posedness.

A first example is given as follows.
It is well-known that a quasisymmetric homeomorphism $\Psi:\R\to\R$ need not be absolutely continuous (see, for instance, \cite[p. 107]{MR1800917}). 
Moreover, by \cite[p. 250]{MR409804}, for each quasisymmetric homeomorphism $\Psi:\R\to\R$ there is a well -posed continuous vector field $b:\R^2\to\R$ 
such that, if $X(t,x)$ denote the (unique) flow $X(t,0,x)$ associated to $b$, then, for some $t>0$, $X(t,\cdot)=\Psi$.
We conclude that there exists a well-posed vector field $b$ 
whose flow is not absolutely continuous at some time.

We provide a second more explicit example from~\cite[Section 8]{MR3984100}.
(Be aware that what is denoted by $t$ in~\cite{MR3984100} is actually the spatial variable.)
In the plane $\R\times\R$ with coordinates $(t,x)$,
consider the family of parabolas of the form $x=a(s)t^2 + s$, where $a:\R\to\R$ is $0$ for $s\le 0$, $1$ for $s\ge 1$ and coincides with the Cantor Staircase Function on $[0,1]$.
These parabolas foliate the plane and are the integral curves of a vector field of the form 
 $(1,b(t,x))$.
One can show that $b\in W^{1,2}_\loc(\R)\cap C^0(\R)\cap {\rm Lip}_\loc(\R\setminus\{0\}))$ and that $(1,b)$ is a well-posed non-autonomous vector field on $\R^2$.\footnote{The uniqueness of integral curves is clear, because $b$ is locally Lipschitz outside the axis $\R\times\{0\}$ and its unique integral curves are then the parabolas  $x=a(s)t^2 + s$.}
However, the flow $X$ of $b$ is not absolutely continuous.
Indeed, for $t\neq0$, the function $X(t,0,\cdot)$ maps the Cantor set to a set of positive measure: with the notation of \cite[page~141]{MR3984100}, if $C\subset\R$ is the Cantor set, then $\Lbm^1(X(t,0,C)) = \Lbm^1(C_t) = t^2/2$.


\subsection{Sub-exponential condition does not imply high Sobolev regularity}\label{sec5eb962e5}
Here we show that
{\it there is a vector field $b\in C^0(\R)$
that satisfies 
\eqref{eq5eb95263}, but 
the flow of $b$ is no better that $W^{1,1}_\loc(\R)$.}
The same example allows us to show that 
{\it 
the upper bound on $\beta$ in 
Proposition~\ref{propFEL}
is necessary for the well-posedness}. 

Given $\alpha\ge1$, let $b:\R\to\R$ be the autonomous vector field on $\R$ given by
\[
b(x) =
\begin{cases}
0&\text{ if }x\le 0\text{ or }x\ge e^{-e},\\
\beta x\log\frac{1}{x}(\log\log\frac{1}{x}-1)^\alpha  &\text{ if }0< x< e^{-e}.
\end{cases} 
\]
We clearly have that
$b\in C^0(\R)\cap W^{1,p}_\loc(\R)$ for all $p\in [1,\infty)$,
 and that $\spt(b) = [0,e^{-e}]$.
Moreover, 
\[
D_xb(x)=
\begin{cases}
0&\text{ if }x< 0\text{ or }x> e^{-e},\\
\beta \log\frac{1}{x}(\log\log\frac{1}{x})^\alpha (1+R(x)) &\text{ if }0< x< e^{-e},
\end{cases} 
\]
where $R$ is a remainder satisfying $\lim_{x\to 0^+}R(x)= 0$.
It is easy to see that
\begin{equation}\label{subexpintdert}
\exp\left(\frac{\left|D_xb\right|}{\left(\log^+\left|D_xb\right|\right)^\alpha}\right)\in \Lbs^1_\loc(\R) .
\end{equation}

Let us consider now the case $\alpha=1$. 
Then $b$ satisfies~\eqref{eq5eb95263} and thus, by  Theorem~\ref{thm5ea9862b}, $b$ is well-posed and the unique flow $X$ of $b$ satisfies $X(t,s)\in W^{1,1}_\loc(\R)$ for every $t,s\in\R$.
In fact, one can explicitly compute $X$ by separation of variables:
for every $t,s \in\R$,
\[
X(t,s,x)= 
\begin{cases}
x&\text{ if }x\le 0,\\
\exp\left(-e \left(\frac{1}{e}\log\frac{1}{x}\right)^{k(t,s)}\right)&\text{ if }0< x< e^{-e},\\
x&\text{ if }x\ge e^{-e},
\end{cases} 
\]
where $k(t,s) = \exp(s-t)$.
The spatial derivative of $X$ for $0< x< e^{-e}$ is
\[
D_xX(t,s,x) = 
\exp\left(-e \left(\frac{1}{e}\log\frac{1}{x}\right)^{k(t,s)}\right) k(t,s)\exp\left(-k(t,s)+1\right)\left(\log\frac{1}{x}\right)^{k(t,s)-1}\frac{1}{x}.
\]
If we now choose $s= 0$ and $t>0$, then it is easy to see that, for every $p>1$,
\[
\int_0^{e^{-e}}|D_xX(t,0,x)|^p dx=+\infty ;
\]
in particular, $X(t,0,\cdot)\notin W^{1,p}_\loc(\R)$ whenever $p>1$. 
Notice also that  $X(t,0,\cdot)\notin C^{0,\gamma}_\loc(\R_x)$ for all $\gamma\in (0,1)$.

In the other case, when $\alpha> 1$, the vector field $b$ still satisfies~\eqref{subexpintdert} and~\eqref{eq5ea991ab}
with $\Theta$ of the form as in~\eqref{sub2expofunct} (but with $\beta>1$).
However, $b$ is not well-posed any more.
For instance, it is easy to see that both functions $\gamma_1\equiv 0$ and
\[
\gamma_2(t)=
\begin{cases}0&\text{if }t\le 0,\\
\displaystyle{\exp\left(-\exp\left(-\left(1+\frac{1}{(\alpha-1)t}\right)^{\frac{1}{\alpha-1}}\right)\right)}&\text{if }t> 0,
\end{cases}
\]
satisfy $\gamma'=b(\gamma)$ and $\gamma(0)=0$.


\subsection{Exponential summability  does not imply the divergence in \texorpdfstring{$BMO$}{BMO} and it is only sufficient  for the Sobolev regularity}\label{sec5ec2a1ee}

Here we show that {\it there is a vector field $b$ satisfying the assumptions of Theorem~\ref{thmbcpcsupp}  such that $\div_x b\notin BMO(\R)$.} 
On the other hand {\it there exist $\ell>0$ and  an exponent $p\in (1,\infty)$ such that, for each $t,\,s\in [-\ell/2,\ell/2]$, 
$X(t,s,\cdot)\in W^{1,p}_{\rm loc}(\R)$, 
even though $\exp(\beta |D_xb|)\notin L^1_{\rm loc}(\R)$ if $\beta:=\,\frac{\ell p^2}{p-1} $}. 

Let $b:\R\to\R$ be the autonomous vector field on $\R$ given by
\[
b(x) =
\begin{cases}
0 & \text{ if }x\le 0\text{ or }x\ge e , \\
x \log\displaystyle{\frac{e}{x}}\, &\text{ if }0<x<e .
\end{cases}
\]
It is easy to see that $b\in C^0(\R)\cap W^{1,p}_\loc(\R)$ for every $p\in [1,\infty)$, and that $\spt(b) =\R\times[0,e]$. 
Moreover, 
\[
D_xb(x) =
\begin{cases}
0&\text{ if }x < 0\text{ or } x > e , \\
\log\frac{1}{x}&\text{ if } 0 < x < e .
\end{cases} 
\]
Let $I:=\, (-\ell/2,\ell/2)$ and $\Omega:=\,(-1,3)$. Observe that, in this case, condition~\eqref{eq5eb965b3} amounts to 
\begin{equation}\label{expintdertex}
\beta:=\,\frac{\ell\,p^2}{p-1} < 1 .
\end{equation}
It is also easy to check (see \cite[Example 7.1.4]{ MR3243741}) that
\[
D_xb(\cdot)\notin BMO(\R)\,.
\]
Using~\eqref{expintdertex} and Theorem~\ref{thmbcpcsupp}, one easily proves that $b:I\times\R\to\R$ is well-posed.
In fact, we can integrate $b$ by separation of variables and obtain
\begin{equation*}
X(t,s,x)=
\begin{cases}
x&\text{ if }x\le 0 , \\
e\left(\frac{x}{e}\right)^{k(t,s)} &\text{ if }0<x< e , \\
x&\text{ if }x\ge e ,
\end{cases} 
\end{equation*}
where $k(t,s)=\exp(s-t)$.
Moreover
\[
D_xX(t,s,x) =
	\begin{cases}
	1 & \text{ if }x < 0 , \\
	k(t,s)\left(\frac{x}{e}\right)^{{k(t,s)-1}} 	&\text{ if }0<x< e , \\
	1 & \text{ if }x > e .
	\end{cases} 
\]
It follows that $X(t,s,\cdot)\in W^{1,q}_\loc(\R)$ if and only if 
\begin{equation}\label{sharpsobregex}
(\exp(s-t)-1)q>-1 ,
\quad\text{i.e.,}\quad
\begin{cases}
q<\frac{1}{1-\exp(s-t)} &\text{ if }s<t ,\\
q\ge1  &\text{ if }s>t .
\end{cases}
\end{equation}
This example shows how the Sobolev regularity of $X(t,0,\cdot)$ can deteriorate with time.
It also shows that condition~\eqref{eq5eb965b3} is only sufficient: 
indeed, if $\ell\ge1/4$ then~\eqref{expintdertex} is not satisfied by any $p$, but $X(t,s,\cdot)\in W^{1,p}_{\rm loc}(\R)$, for each $t,\,s\in [-\ell/2,\ell/2]$,  if $1<\,p<\, \frac{1}{1-\exp(-\ell)}$. Thus Theorem~\ref{thmbcpcsupp} does not apply. Notice also that, by~\eqref{sharpsobregex}, if $t>\,0$, $X(t,0,\cdot)\notin W^{1,q}_{\rm loc}(\R)$   for any $q\ge\,\frac{1}{1-\exp(-t)}$.

\printbibliography

\end{document}